\documentclass[11pt,reqno,tbtags]{amsart}
\usepackage{amssymb}
\usepackage{url}
\usepackage[square,numbers]{natbib}
\bibpunct[, ]{[}{]}{;}{n}{,}{,}

\title
{Patterns in random permutations avoiding the pattern 132}

\date{22 January, 2014}

\author{Svante Janson}
\thanks{Partly supported by the Knut and Alice Wallenberg Foundation}
\address{Department of Mathematics, Uppsala University, PO Box 480,
SE-751~06 Uppsala, Sweden}
\email{svante.janson@math.uu.se}
\newcommand\urladdrx[1]{{\urladdr{\def~{{\tiny$\sim$}}#1}}}
\urladdrx{http://www2.math.uu.se/~svante/}


\numberwithin{equation}{section}

\renewcommand\le{\leqslant}
\renewcommand\ge{\geqslant}

\allowdisplaybreaks

\newtheorem{theorem}{Theorem}[section]
\newtheorem{lemma}[theorem]{Lemma}

\newtheorem{corollary}[theorem]{Corollary}

\theoremstyle{definition}
\newtheorem{example}[theorem]{Example}

\newtheorem{remark}[theorem]{Remark}

\theoremstyle{remark}

\newenvironment{romenumerate}[1][0pt]{
\addtolength{\leftmargini}{#1}\begin{enumerate}
 }{\end{enumerate}}

\newcounter{oldenumi}
{\setcounter{oldenumi}{\value{enumi}}
\begin{romenumerate} \setcounter{enumi}{\value{oldenumi}}}
{\end{romenumerate}}

\newcounter{thmenumerate}
\newenvironment{thmenumerate}
{\setcounter{thmenumerate}{0}%
 \def\item{\par
 \refstepcounter{thmenumerate}\textup{(\roman{thmenumerate})\enspace}}
}
{}

\newcounter{romxenumerate}   

\newcounter{xenumerate}   

\newcommand\pfitem[1]{\par(#1):}
\newcommand\pfitemx[1]{\par#1:}
\newcommand\pfitemref[1]{\pfitemx{\ref{#1}}}

\newcommand{\refT}[1]{Theorem~\ref{#1}}
\newcommand{\refC}[1]{Corollary~\ref{#1}}
\newcommand{\refL}[1]{Lemma~\ref{#1}}
\newcommand{\refR}[1]{Remark~\ref{#1}}
\newcommand{\refS}[1]{Section~\ref{#1}}

\newcommand{\refE}[1]{Example~\ref{#1}}

\newcommand\REM[1]{{\raggedright\texttt{[#1]}\par\marginal{XXX}}}

\newenvironment{comment}{\setbox0=\vbox\bgroup}{\egroup} 

\begingroup
  \divide\count255 by 60
  \multiply\count255 by -60
  \advance\count255 by \time
  \ifnum \count255 < 10 \xdef\klockan{\the\count1.0\the\count255}
  \else\xdef\klockan{\the\count1.\the\count255}\fi
\endgroup

\DeclareMathOperator*{\sumsumsum}{\sum\sum\sum}

\newcommand{\sumni}{\sum_{n=1}^\infty}

\newcommand\set[1]{\ensuremath{\{#1\}}}

\newcommand\biggset[1]{\ensuremath{\biggl\{#1\biggr\}}}

\newcommand\xpar[1]{(#1)}
\newcommand\bigpar[1]{\bigl(#1\bigr)}
\newcommand\Bigpar[1]{\Bigl(#1\Bigr)}

\newcommand\lrpar[1]{\left(#1\right)}

\newcommand\bigabs[1]{\bigl|#1\bigr|}
\newcommand\Bigabs[1]{\Bigl|#1\Bigr|}

\def\rompar(#1){\textup(#1\textup)}    

\newcommand\parfrac[2]{\lrpar{\frac{#1}{#2}}}

\newcommand\Bigparfrac[2]{\Bigpar{\frac{#1}{#2}}}

\def\xexp(#1){e^{#1}}

\newcommand\floor[1]{\lfloor#1\rfloor}

\newcommand\ntoo{\ensuremath{{n\to\infty}}}

\newcommand\rtoo{\ensuremath{{r\to\infty}}}

\newcommand\downto{\searrow}

\newcommand\punkt{.\spacefactor=1000}    
    
\newcommand\ie{i.e\punkt}
\newcommand\eg{e.g\punkt}
\newcommand\viz{viz\punkt}
\newcommand\cf{cf\punkt}
\newcommand{\as}{a.s\punkt}

\newcommand{\tend}{\longrightarrow}
\newcommand\dto{\overset{\mathrm{d}}{\tend}}
\newcommand\pto{\overset{\mathrm{p}}{\tend}}

\newcommand\eqd{\overset{\mathrm{d}}{=}}

\newcommand\bbR{\mathbb R}
\newcommand\bbC{\mathbb C}

\newcounter{CC}
\newcounter{cc}

\renewcommand\Re{\operatorname{Re}}

\newcommand\E{\operatorname{\mathbb E{}}}
\newcommand\PP{\operatorname{\mathbb P{}}}
\newcommand\Var{\operatorname{Var}}

\newcommand\Bi{\operatorname{Bi}}

\newcommand\ga{\alpha}
\newcommand\gb{\beta}
\newcommand\gd{\delta}
\newcommand\gD{\Delta}

\newcommand\gG{\Gamma}

\newcommand\gl{\lambda}
\newcommand\gL{\Lambda}
\newcommand\go{\omega}

\newcommand\gs{\sigma}

\newcommand\gt{\tau}

\renewcommand\phi{\xxx}  

\newcommand\cB{\mathcal B}

\newcommand\cL{{\mathcal L}}

\newcommand\cP{\mathcal P}

\newcommand\qw{^{-1}}
\newcommand\qww{^{-2}}
\newcommand\qq{^{1/2}}
\newcommand\qqw{^{-1/2}}

\newcommand\qqqq{^{1/4}}
\newcommand\qqc{^{3/2}}

\newcommand\intoi{\int_0^1}

\newcommand\oi{[0,1]}

\newcommand\dd{\,\mathrm{d}}

\newcommand{\mgf}{moment generating function}
\newcommand{\gfx}{generating function}

\newcommand\lhs{left-hand side}
\newcommand\rhs{right-hand side}

\newcommand\fS{\mathfrak S}
\newcommand\fSn{\fS_n}
\newcommand\fSx{\fS_*}
\newcommand\fSnzzz{\fS_n(\zzz)}
\newcommand\fSkzzz{\fS_k(\zzz)}
\newcommand\fSxzzz{\fSx(\zzz)}
\newcommand\nt{n_\gt}
\newcommand\ns{n_\gs}
\newcommand\pint{\pinx{\gt}}
\newcommand\pinx[1]{\boldsymbol{\pi}_{#1,n}}
\newcommand\pinzzz{\pinx{\zzz}}
\newcommand\zzz{132}
\newcommand\tr{^\dagger}
\newcommand\xbar{\overline}
\newcommand\ed{\E_\gd}
\newcommand\td{T_\gd}
\newcommand\tn{T_n}
\newcommand\gdx[1]{\gd^{-#1}}
\newcommand\ogdx[1]{O\bigpar{\gd^{-#1}}}
\newcommand\gdxx[1]{\gdx{(#1)}}
\newcommand\ogdxx[1]{O\bigpar{\gdxx{#1}}}

\DeclareMathOperator*{\sumsum}{\sum\sum}
\newcommand\gc{\gamma}
\newcommand\gxx[1]{\gG\bigpar{(#1)/2}}
\newcommand\nxx[1]{n^{-(#1)/2}}
\newcommand\gsxx[2]{\gs_{#1}\dotsm\gs_{#2}}
\newcommand\ee{B}
\newcommand\eegs{\ee_\gs}
\newcommand\gspp[1]{\gs^{(#1)}}
\newcommand\ex{\mathbf{e}}
\newcommand\coi{C[0,1]}
\newcommand\vv{\bar v}
\newcommand\tdx{\tilde h}
\newcommand\tdl{\tilde h_L}
\newcommand\tf{\tilde f}
\newcommand\Phix{\widetilde\Phi}
\newcommand\Psis{\Psi_\gs}
\newcommand\gsik{1\dotsm k}
\newcommand\gski{k\dotsm 1}
\newcommand\df{depth first}
\newcommand\dfw{depth first walk}
\newcommand\cex{2\qqc \ex}
\newcommand\clex{2\qq\ex}
\newcommand\txfrac[2]{\tfrac{#1}{#2}}

\newcommand{\Holder}{H\"older}

\begin{document}

\begin{comment}  
05 Combinatorics 
05A Enumerative combinatorics [For enumeration in graph theory, see 05C30]
05A05 Permutations, words, matrices 

60 Probability theory and stochastic processes
60C Combinatorial probability
60C05 Combinatorial probability

60F Limit theorems [See also 28Dxx, 60B12]
60F05 Central limit and other weak theorems
60F17 Functional limit theorems; invariance principles

\end{comment}

\begin{abstract} 
We consider a random permutation drawn from
the set of 132-avoiding permutations of length $n$ and show that
the number of occurrences of another pattern $\sigma$ has a limit distribution,
after scaling by $n^{\lambda(\sigma)/2}$ where $\lambda(\sigma)$ is the
length of $\sigma$ plus the number of descents.
The limit is not normal, and can be expressed as a functional of a Brownian
excursion. Moments can be found by recursion.
\end{abstract}

\maketitle

\section{Introduction}\label{S:intro}

We say that two sequences 
(of the same length)
$x_1\dotsm x_k$ and $y_1\dotsm y_k$ 
of real numbers  have \emph{the same order} if
$x_i<x_j \iff y_i<y_j$ for all $i,j\in [k]$.

Let $\fS_n$ be the set of permutations of $[n]:=\set{1,\dots,n}$.
If $\gs=\gs_1\dotsm\gs_k\in\fS_k$ and $\pi=\pi_1\dotsm\pi_n\in\fS_n$,
then an \emph{occurrence} of $\gs$ in $\pi$ is a subsequence
$\pi_{i_1}\dotsm\pi_{i_k}$, with $1\le i_1<\dots<i_k\le n$,
that has the same order as $\gs$.
We let $\ns(\pi)$ be the number of occurrences of $\gs$ in $\pi$, and note
that
\begin{equation}\label{11}
  \sum_{\gs\in\fS_k} \ns(\pi) = \binom nk,
\end{equation}
for every $\pi\in\fS_n$.

We say that $\pi$ \emph{avoids} $\gs$ if $\ns(\pi)=0$; otherwise, $\pi$
\emph{contains} $\gs$. Let 
\begin{equation}
\fS_n(\gs):= \set{\pi\in\fS_n:\ns(\pi)=0},
\end{equation}
the set 
of permutations of length $n$ that avoid $\gs$.
We also let $\fSx(\gs):=\bigcup_{n=1}^\infty\fS_n(\gs)$ be the set of
$\gs$-avoiding permutations of arbitrary length. 

\begin{remark}\label{Rsymm}
  For later use, note that 
$n_{\gs\qw}(\pi\qw)=n_\gs(\pi)$.
Similarly, for the reverse $\gs\tr:=\gs_k\dots\gs_1$,
$n_{\gs\tr}(\pi\tr)=n_\gs(\pi)$, and for the complement
$\xbar\gs=(k+1-\gs_1)\dotsm(k+1-\gs_k)$,
$n_{\xbar\gs}(\xbar\pi)=n_\gs(\pi)$.
In particular, the maps $\pi\mapsto\pi\qw$,
$\pi\mapsto\pi\tr$ and $\pi\mapsto\xbar\pi$
are bijections $\fS_n(\gs)\to \fS_n(\gs\qw)$,
$\fS_n(\gs)\to \fS_n(\gs\tr)$
and
$\fS_n(\gs)\to \fS_n(\xbar\gs)$.
\end{remark}

The general problem that we are interested in here is to
take a fixed permutation 
$\gt$,  
and let $\pint$ be a uniformly random $\gt$-avoiding permutation, \ie, a
uniformly random element of $\fS_n(\gt)$, and then study the 
distribution of the random variable $\ns(\pint)$  for some other
fixed permutation $\gs$.
More precisely, we are mainly interested in 
asymptotics of the distribution as \ntoo{}. 
(Although our methods also yield  exact formulas for finite $n$.)
The present paper is only a partial contribution to this general problem, and we
will soon concentrate on the single case $\gt=132$.

\begin{remark}
It is well-known that if $\pi$ is a uniformly random permutation in $\fS_n$,
without any restriction,
and $\gs$ is a fixed permutation,
then $\ns(\pi)$ has an asymptotic normal distribution as \ntoo; moreover,
this holds jointly for several $\gs$. See \citet{Bona-Normal,Bona3}
and \citet{SJ287}. We shall see that the restricted case is different.
\end{remark}

\begin{remark}
The case $|\gt|=2$ is trivial. By symmetry (\refR{Rsymm}), it suffices to
consider 
$\gt=21$, and then $\nt(\pi)$ is the number of inversions in $\pi$; the only
permutation in $\fS_n$ that avoids 21 is the identity permutation so 
$\fS_n(21)=\set{12\dotsm n}$ has only one element.
Hence, the simplest non-trivial cases are the cases $|\gt|=3$. There are 6
permutations  $\gt\in\fS_3$, but by the symmetries in \refR{Rsymm}, it
suffices to consider the two cases $\gt=123$ and $132$.
\end{remark}

As a background, note first that it is a classical problem to enumerate the sets
$\fS_n(\gt)$, either exactly or asymptotically,  and to study various properties
of the generating function;
see \citet[Chapters 4--5]{Bona}.
In particular, two permutations $\gs$ and $\gt$
are said to be \emph{Wilf-equivalent} if $|\fS_n(\gs)|=|\fS_n(\gt)|$ for all
$n$. It is know that all permutations of length 3  are Wilf-equivalent, 
with $|\fS_n(\gt)|=\binom{2n}{n}/(n+1)$, the $n$th Catalan number $C_n$,
when $|\gt|=3$, see \eg{} 
\cite[Exercises 2.2.1-4]{KnuthI}, 
\cite{Simion-Schmidt},
\cite[Exercise 6.19ee,ff]{StanleyII},
\cite[Corollary 4.7]{Bona};
in contrast, 
not all permutations of length 4 are Wilf-equivalent. 
(The classification of Wilf-equivalent
permutations of length 4 was quite difficult, see \cite{Bona} and the
references given there.)

A simpler version of the general 
problem above is to find (at least asymptotically)
the expectation $\E\ns(\pint)$.
(If the number $|\fS_n(\gt)|$ is known, this is equivalent to finding the
total number of occurrences of $\gs$ in all $\gt$-avoiding permutations of
length $n$.)
This version of the problem was posed by \citet{Cooper},
and  has been studied by
\citet{Bona-abscence} ($\gt=132$, $\gs=1\dotsm k$ and $\gs=k\dotsm1$),  
\citet{Bona-surprising} ($\gt=132$, $|\gs|=3$ and certain longer $\gs$),  
\citet{Homberger} ($\gt=123$, $|\gs|\le3$);
furthermore \citet{CEF} studies the case $\gt=321$, $\gs=21$.
(or, equivalently, $\gt=123$, $\gs=12$).
These papers concentrate on exact formulas and generating functions;
asymptotics are derived as corollaries.
 \citet{Rudolph} studied the problem of when 
$\E n_{\gs_1}(\pint)=\E n_{\gs_2}(\pint)$ (in the case $\gt=132$).

In particular, for $\tau=132$,
by \cite{Bona-abscence},  \cite{Bona-surprising} 
and straightforward singularity analysis 
(see \cite[Chapter VI]{FS}), 
or by Examples \ref{Eed3} and \ref{Eed1...k}
below, as \ntoo,
\begin{align}
  \E n_{12}(\pinx{132}) 
&\sim \frac{\sqrt\pi}2 n^{3/2}, \label{e12}
\\
  \E n_{123}(\pinx{132}) 
&\sim \frac12 n^2, \label{e123}
\\
  \E n_{213}(\pinx{132})
&=  \E n_{231}(\pinx{132})
=  \E n_{312}(\pinx{132})
\sim \frac{\sqrt\pi}{8} n^{5/2},
\label{e213}\\
  \E n_{321}(\pinx{132}) 
&\sim\binom n3 \sim \frac16 n^3, \label{e321}
\intertext{and, for any fixed $k\ge1$,
generalizing \eqref{e12}--\eqref{e123},}
  \E n_{1\dotsm k}(\pinx{132}) 
&\sim \frac{2^{1-k}\sqrt\pi}{\gG(k/2)} n^{(k+1)/2}. \label{e1...k}
\end{align}
Note that in \eqref{e213}, the three expectations are equal for any $n$; the
equality of the two latter is trivial because
$n_{231}(\pinx{132})$
and $n_{312}(\pinx{132})$ have the same distribution, as a consequence of
the first symmetry in \refR{Rsymm}. 
The first equality is non-trivial and more surprising; in fact
$n_{213}(\pinx{132})$ and $n_{231}(\pinx{132})$ do not have the same
distribution, in general. 
(They have different variances already for $n=5$, as is shown by
an enumeration, by hand or by  computer.)

The more general problem of studying the distribution, and not just the
expectation, of $\ns(\pint)$ was raised in \cite{SJ287}, where higher
moments (and mixed moments) 
are calculated (using computer algebra) for small $n$ for
several cases ($\gt= 132$, 123 and 1234;  several $\gs$ with $|\gs|=3$).

The main result of the present paper 
(\refS{Sresults})
is that the formulas above for the
expectation generalize to arbitrary $\gs\in\fSxzzz$, always with growth as a
half-integer power of $n$, and that, moreover, the random variables after
normalization by this power of $n$ converge to some positive
limit random variables, with convergence of
all moments.

\begin{remark}
The case of forbidding $\tau=123$ has, as said above, been studied by 
 \citet{CEF} ($\gs=12$) and
\citet{Homberger} ($|\gs|\le3$); their results
yield (after simple calculations and corrections of several typos in
\cite{Homberger}), as \ntoo, 
\begin{align}
  \E n_{12}(\pinx{123}) 
&\sim \frac{\sqrt\pi}4 n^{3/2}, \label{e12H}
\\
  \E n_{132}(\pinx{123}) 
&=
  \E n_{213}(\pinx{123})
\sim \frac14 n^2, \label{e132H}
\\
\E n_{231}(\pinx{123})
&=  \E n_{312}(\pinx{123})
\sim \frac{\sqrt\pi}{8} n^{5/2}, \label{e231H}
\intertext{and, which also follows from these and \eqref{11},}
  \E n_{321}(\pinx{123}) 
&\sim\binom n3 \sim \frac16 n^3. \label{e321H}
\end{align}
Cf.~\eqref{e12}--\eqref{e321}. 
Moreover,
  \citet{Homberger} shows that 
also
$\E n_{231}(\pinx{123})
=  \E n_{231}(\pinx{132})$ for any $n$;
however, the distribution of
$n_{231}(\pinx{123})$ differs (in general) 
from the distribution of any of the variables in \eqref{e213}.
(They have different variances already for $n=4$.)

The equivalence given by \cite{CEF} between $n_{12}(\pinx{123})$ and the
number of certain squares under a Catalan path (or equivalently, a Dyck path)
implies by standard results that
\begin{equation}
  n\qqw n_{12}(\pinx{123})\sim 2\qqw \intoi \ex(x)\dd x
\end{equation}
where $\ex$ is a Brownian excursion; this is apart from a factor $1/2$ the
same limit as for $n_{12}(\pinx{132})$, see \refE{Epsi12}.
For the other cases above (excluding the trivial $n_{321}$) we do not know
any asymptotic distribution, and not even asymptotic second moments. 
It seems likely that methods similar to the present paper could be useful in
this case too, using a suitable bijection between $\fS_n(123)$ and binary
trees (\cf{} \refR{Rbij}), but we have not yet attempted it.

It seems much more difficult to show results for any longer $\tau$.
\end{remark}

\begin{remark}\label{R0}
  A special case of the distribution of $\ns(\pint)$ is the probability
  $\PP\bigpar{\ns(\pint)=0}$ that a $\tau$-avoiding permutation also avoids
	$\gs$;
this is equivalent to enumerating the set $\fSn(\gs,\tau)$ of permutations
that avoid both $\gs $ and $\tau$ (given that we know $|\fSn(\tau)|$).
This problem has been studied by various authors (with exact results,
generating functions and asymptotics), see \eg{}
\cite{Simion-Schmidt}, 
\cite{BJS},
\cite{West-forbidden},
\cite{Chow-West},
\cite{RobertsonWZ},
\cite{MV-chebyshev}, \cite{MV-132}, \cite{MV-Lothar},
\cite{Krattenthaler},
\cite{AlbertAB2011},
\cite{AlbertAB2012}.
Some of these also consider the number of $\tau$-avoiding permutations with
exactly $r$ occurences of $\gs$, which is equivalent to $\PP(\ns(\pint)=r)$.
Formally, this is the same as our problem of the distribution of
$\ns(\pint)$, but the emphasis in these papers is on exact formulas for
constant $r$, while we are interested in asymptotic results, with $r$
increasing. It would be interesting to derive asymptotic distributions from
these algebraic results, but this seems difficult.
\end{remark}

\begin{remark}\label{Rflera}
  We have considered avoiding a single pattern $\tau$. Of course, the
  same questions can be asked for a set $\tau_1,\dots,\tau_M$ of two or
  several forbidden patterns, \cf{} the references in \refR{R0} where such
  sets $\fSn(\tau_1,\dots,\tau_M)$ are studied. For a simple example,
there are exactly $2^{n-1}$ permutations in $\fSn(123,132)$, and they have a
simple structure \cite{Simion-Schmidt}
which makes it easy to see that the number $n_{12}$ of noninversions has
a binomial distribution $\Bi(n-1,1/2)$; in this case, $n_{12}$  thus has
an asymptotically normal distribution.
\end{remark}

\section{Main results}\label{Sresults}
From now on we consider only $\gt=132$.
Note that $\ns(\pinzzz)=0$ if $\gs$ contains a copy of $\gt$; hence we only
consider $\gs$ that  themselves avoid $\gt$.

Recall that a \emph{descent} in a permutation $\gs_1\dotsm\gs_k$ is an index
$i\in[k-1]$ such that $\gs_i>\gs_{i+1}$; we also define the last index $k$
to be a decent. (Tradition varies about the latter case; we find this
version convenient for our purposes.) We let $D(\gs)$ be the number of
descents in $\gs$. (Note that with our definition $1\le D(\gs)\le|\gs|$.)
We define 
\begin{equation}\label{gl}
  \gl(\gs):=|\gs|+D(\gs)
\end{equation}
and note that 
\begin{equation}
|\gs|+1\le\gl(\gs)\le 2|\gs|,  
\end{equation}
with the extreme values $\gl(\gs)=|\gs|+1$ if and only if $\gs=1\dotsm k$,
and
$\gl(\gs)=2|\gs|$ if and only if $\gs=k\dotsm 1$, where $k=|\gs|$.

\begin{theorem}\label{Tmain}
There exist 
strictly positive random variables $\gL_\gs$
such that
\begin{equation}\label{tmaind}
 \ns(\pinzzz)/ n^{\gl(\gs)/2} \dto \gL_\gs,
\end{equation}
as \ntoo, jointly for all $\gs\in\fSxzzz$. 
Moreover, this holds with convergence of
all moments, with all moments of $\gL_\gs$ finite, \ie, for any sequence
$\gspp1,\dots,\gspp{M}\in\fSxzzz$, possibly with repetitions,
\begin{equation}\label{tmainmom}
  \E\bigpar{n_{\gspp 1}\dotsm n_{\gspp M}(\pinzzz)}
\sim n^{\sum_\nu\gl(\gspp\nu)/2} \E\bigpar{ \gL_{\gspp1}\dotsm\gL_{\gspp M}}.
\end{equation}
In particular,
for every $\gs\in\fSxzzz$, there exists a positive constant
$A_\gs=\E\gL_\gs$ 
such that
\begin{equation}\label{tmaine}
  \E \ns(\pinzzz) \sim A_\gs n^{\gl(\gs)/2}.
\end{equation}

For a monotone decreasing permutation $k\dotsm1$,  
$\gL_{k\dotsm 1}=1/k!$ is deterministic, but not for any other $\gs$.
\end{theorem}

\begin{remark}
Since $\gL_\gs>0$, the limit distributions are \emph{not} normal;
thus $n_\gs(\pinzzz)$ is not asymptotically normal.
(For  $\gs=\gski$, use \eqref{rk1} below.)  
This was conjectured (for $\gs=312$) in \cite{SJ287} based on calculation of
the moments for small $n$; our theorem verifies this, but it should be noted
that the numerical values in \cite[Table 3]{SJ287} for $n\le20$ are still
far from their limits. A calculation using \refT{Tmom3} shows that 
the normalized third moment 
$\E(X-\E X)^3/\Var(X)^{3/2}\approx 0.76384$  
for the limit  $X=\gL_{312}$,
while for $n=20$, \cite{SJ287} yields $0.44906$.
\end{remark}

The proof of \refT{Tmain} will occupy the rest of the paper. 
We will use two completely different methods that complement each other and
prove different parts of the theorem; both use a bijection with binary trees
described in \refS{Sbinary}.
One method (\refS{Sbrown} and \refT{Tpsi}) uses this to show the convergence
in distribution \eqref{tmaind}; this proof shows also
that the limit random variables
$\gL_\gs$ can be expressed as functionals of a Brownian excursion $\ex(x)$.
In particular (\refE{Epsi12}), $\gL_{12}=\sqrt2\intoi\ex(x)\dd x$; this is
(apart from the factor $\sqrt2$) 
the well-known \emph{Brownian excursion area} which appears as a limit in
various 
combinatorial problems
(for instance for the total path length in a random conditioned
Galton--Watson tree  \cite{AldousII,AldousIII});
for this distribution see also the survey
\cite{SJ201} and the references there.
(It is sometimes called the \emph{Airy distribution}.)
More generally (\refE{Epsi1...k}), for the monotone pattern $\gsik$, 
$\gL_{\gsik}=c_k\intoi \ex(x)^{k-1}\dd x$ with $c_k=2^{(k-1)/2}/(k-1)!$.
However, in general, the description as a Brownian excursion functional is
rather complicated, and it is not easy to even compute its mean.

As a complement, we therefore give also by another method (\refS{Sexp})
formulas  
yielding (by recursion) the constants $A_\gs=\E \gL_\gs$, see
\eqref{ags} and \eqref{eegs}; we describe also 
(\refS{Shigher}) how one
can similarly find also limits for higher moments (possibly mixed).
This method uses a recursion for the numbers $\ns(\pi)$ that is given in
\refS{Srec}, 
and a probabilistic argument using subcritical Galton--Watson trees.
As examples, we give (\refT{Tmom3}) explicit
recursion relations for the moments of
$\gL_\gs$ for $|\gs|\le3$ (and joint moments of $\gL_{12}$ and $\gL_\gs$
with $|\gs|=3$, needed for the recursions).
In particular, \refT{Tmom3} yields for the second moments
(where \eqref{e2gl12} is well-known, see \cite{Louchard}, \cite{SJ201})
\begin{align}
  \E \gL_{12}^2&=\frac56 \label{e2gl12},
&
\Var \gL_{12} &= \frac{10-3\pi}{12},
\\
  \E \gL_{123}^2&=\frac{19}{60},
&
\Var \gL_{123} &= \frac{1}{15},
\\
  \E \gL_{213}^2&=\frac{7}{120},
&
\Var \gL_{213} &= \frac{56-15\pi}{960},
\\
  \E \gL_{231}^2=
  \E \gL_{312}^2
&=\frac{43}{840},
&
\Var \gL_{231} =
\Var \gL_{312} 
&= \frac{344-105\pi}{6720}.
\end{align}
For mixed moments we find from \refT{Tmom3} for example
\begin{align}
  \E\bigpar{\gL_{12}\gL_{213}} &= \frac{13}{60},
\\
  \E\bigpar{\gL_{12}\gL_{231}} =
  \E\bigpar{\gL_{12}\gL_{312}} &= \frac{1}{5}.
\end{align}
The matrix of second moments of $(\gL_{213},\gL_{231},\gL_{312})$ is given in
\eqref{ecov3}.

\begin{remark}\label{Rbonadom}
  For a given $|\gs|=k$, we see that the order of $\E \ns(\pinzzz)$ is 
smallest ($n^{(k+1)/2}$) for $\gs=1\dotsm k$
and largest ($n^k$)  for $\gs=k\dotsm 1$.
Cf.\ the related result by \citet{Bona-abscence} that for every $n$,
$\E n_{1\dotsm k}(\pinzzz) \le \E n_{\gs}(\pinzzz)
\le \E n_{k\dotsm 1}(\pinzzz)$ for every $\gs\in\fSkzzz$, see \refS{Sfurther}.	
\end{remark}

\begin{remark}\label{Rk1}
In particular, \eqref{tmaine} implies that $\E \ns(\pinzzz)/n^k\to0$ for
every $\gs\in\fS_k$ except $k\dotsm1$, which by \eqref{11} trivially implies
$\E n_{k\dotsm1}(\pinzzz)\sim \binom nk$ and 
$n_{k\dotsm1}(\pinzzz)\pto 1/k!$, which is the case $\gs=k\dotsm1$ of
\refT{Tmain} with $\gL_{k\dotsm1}=1/k!$ deterministic as asserted in the
theorem.

For a nondegenerate limit law also in this case (for $k>1$), 
note that the same argument
yields
\begin{equation}\label{rk1}
n^{-(k-1/2)} \lrpar{\binom nk -n_{k\dotsm1}(\pinzzz)}
\dto \sum \gL_\gs,
\end{equation}
summing over all $\gs\in\fSkzzz$ with $\gl(\gs)=2k-1$ (i.e., $D(\gs)=k-1$).
\end{remark}

\begin{remark}
  Although the exponent in \eqref{tmaine} depends only on $\gl(\gs)$, \ie,
  on $|\gs|$ and the number of  descents in $\gs$, the constant $A_\gs$
  does not. For example, it follows by 
\eqref{ags} and \eqref{eegs}, or by \refE{Eed4} and \refL{Lsing}, that
$\E n_{3214}\sim \frac{\sqrt\pi}{32}n^{7/2}$ and
$\E n_{3241}\sim \frac{\sqrt\pi}{64}n^{7/2}$.
\end{remark}

\begin{remark}
  Apart from the relation \eqref{11}, there are also simple relations between
  the counts $n_\gs(\pi)$ for $\gs$ of different lengths. For example, 
\begin{equation}\label{bysans}
 (n-2) n_{12}(\pi)
=3n_{123}(\pi)+ 2n_{132}(\pi)+ 2n_{213}(\pi)+n_{231}(\pi)+n_{312}(\pi),
\end{equation}
since the \lhs{} counts the number of distinct $i,j,k$ such that $i<j$ and
$\pi_i<\pi_j$, and if $\gs\in\fS_3$, then each occurence of $\gs$ in $\pi$ 
contributes $n_{12}(\gs)$ such triples.

For $\pi\in\fSn(132)$, the term $n_{132}(\pi)$ vanishes, and if we divide by
$n^{5/2}$ and take the limit, another term disappears asymptotically, and we
find for the limit variables the relation
\begin{equation}\label{rom}
  \gL_{12}
= 2\gL_{213}+\gL_{231}+\gL_{312}.
\end{equation}
Similar relations enable each $\gL_\gs$ to be expressed in $\gL_{\gs'}$ for
some set of $\gs'$ with $|\gs'|=|\gs|+1$.
\end{remark}

\begin{remark}
  The limit $\gL_{213}$ and the sum $\gL_{231}+\gL_{312}$ have appeared
  earlier as  distribution limits in \cite{SJ146}, see \refR{Rwiener}.
\end{remark}

\section{A basic recursion}\label{Srec}

If $x_1\dotsm x_n$ is any sequence of distinct numbers, let $\Pi(x_1\dotsm
x_n)$ be the permutation in $\fS_n$ that has the same order as $x_1\dotsm x_n$.
We extend the notation $n_\gs(\pi)$ in the trivial way to arbitrary
sequences of distinct numbers $x_1\dotsm x_n$ and $y_1\dotsm y_k$ by
$n_{y_1\dotsm y_k}(x_1\dotsm x_n)
:= n_{\Pi(y_1\dotsm y_k)}(\Pi(x_1\dotsm x_n))$.
(We may similarly extend other notations when convenient.)
We also define
$n_\emptyset(x_1\dotsm x_n)=0$ for an empty string
$\emptyset$
(\ie, the case $k=0$), and let $\fS_0:=\set{\emptyset}$.

If $\pi\in\fS_n$ and $\ell$ is the index of the maximal element $n$,
\ie. $\pi_\ell=n$, let $\pi_L:=\pi_1\dotsm\pi_{\ell-1}$ and
$\pi_R:=\pi_{\ell+1}\dotsm\pi_n$ be the (possibly empty)
parts of $\pi$ before and after the maximal element.
Using the operator $\Pi$ above, we can regard them as permutations
$\pi_L\in\fS_{\ell-1}$ and $\pi_R\in\fS_{n-\ell}$.

We begin with a well-known characterization of the 132-avoiding
permutations,
see \eg{} \citet{Bona-abscence}.
\begin{lemma}
  \label{L132}
With notations as above, a permutation $\pi$ avoids $132$ if and only if
$\pi_L$ and $\pi_R$ both avoid $132$ and furthermore
$\pi_i>\pi_j$ whenever $i<\ell$ and $j>\ell$.
\end{lemma}
\begin{proof} 
Although this is well-known and easy, we sketch the proof for completeness.

  If $\pi$ avoids 132 then so do $\pi_L$ and $\pi_R$. Furthermore, if the
  final condition in the lemma is violated, then $\pi_i<\pi_j<\pi_\ell$ for
  some $i$ and $j$ with $i<\ell<j$, and thus $\pi_i\pi_\ell\pi_j$ is an
  occurrence of 132.

The converse is just as easy, by considering the possible positions of an
occurrence of 132 in relation to $\ell$; we omit the details.
\end{proof}

This leads to a basic recursion for $\ns(\pi)$.
\begin{lemma}
  \label{LR}
Let $\gs\in\fS_k(132)$ with $k\ge1$.
Define $m$ by $\gs_m=k$ and let 
$\gD:=\set{q\in[k-1]:\min_{1\le i\le q}\gs_i>\max_{q<j\le k}\gs_j}$.
Then, for any permutation $\pi\in\fSnzzz$ with $n\ge1$,
\begin{multline}\label{lr}
\ns(\pi)
=
\ns(\pi_L)+\ns(\pi_R)+\sum_{q\in \gD}
 n_{\gsxx1q}(\pi_L)n_{\gsxx{q+1}k}(\pi_R)
\\
+n_{\gsxx1{m-1}}(\pi_L)n_{\gsxx{m+1}k}(\pi_R).
 \end{multline}
\end{lemma}

\begin{proof}
  Consider first an occurrence $\pi_{\nu_1}\dotsm\pi_{\nu_k}$ of $\gs$ that does
  not include $\pi_\ell$. 
Then, for some $q\in\set{0,\dots,k}$,
$\nu_1<\dotsm <\nu_q<\ell < \nu_{q+1}<\dotsm <\nu_k$.

The cases $q=k$ and $q=0$ give the $\ns(\pi_L)$ and $\ns(\pi_R)$ occurrences
in $\pi_L$ and $\pi_R$. 

If $1\le q\le k-1$, we note that by \refL{L132}, if $i\le q$ and $j>q$, then
$\pi_{\nu_i}>\pi_{\nu_j}$ and thus $\gs_{i}>\gs_{j}$; hence $q\in\gD$.
Furthermore, for every $q\in\gD$, we have excatly one such occurrence
$\gs$ in $\pi$ for every pair of occurrences 
of $\gsxx{1}{q}$ in $\pi_L$ and $\gsxx{q+1}k$ in $\pi_R$. The total number
of such occurrences is thus the sum in \eqref{lr}.

Finally, if an occurrence $\pi_{\nu_1}\dotsm\pi_{\nu_k}$ of $\gs$ contains
$\pi_\ell=n$, then $\pi_\ell$ must correspond to the largest element $\gs_m$
in $\gs$, \ie{} $\nu_m=\ell$.
It follows in the same way as above that the number of such occurrences is 
$n_{\gsxx1{m-1}}(\pi_L)n_{\gsxx{m+1}k}(\pi_R)$.
\end{proof}

The set $\gD$ is empty if $m=k$; otherwise $m\in\gD$ by \refL{L132} so
$\gD\neq\emptyset$. The extreme case is $\gs=k\dotsm1$ when $\gD=[k-1]$.
Note that every element of $\gD$ is a descent in $\gs$ (but not conversely,
in general).

\section{Binary trees}\label{Sbinary} 

Out proofs are based on a well-known bijection between $\fS_{n}(132)$ and
the set $\cB_n$ of binary trees of order $n$, see \eg{} \cite{Bona-surprising}.
It can be defined as follows.

Recall that a binary tree $T$ consist of
a root and two subtrees $T_L$ and $T_R$ (the \emph{left} and \emph{right}
subtree) which are either empty or themselves binary trees.
Using the notations of \refS{Srec}, we define recursively for any
permutation $\pi\in\fS_n(132)$ with $n\ge1$ a binary tree 
$T=T(\pi)\in\cB_n$ such that its left subtree $T_L=T(\pi_L)$ and its right
subtree $T_R=T(\pi_R)$; furthermore, $T(\emptyset)$ is the empty tree.
It is easy to see that this yields a bijection between $\fS_n(132)$ and $\cB_n$.

If $T$ is a binary tree, and $\gs$ is a permutation, let
$X_\gs(T):=\ns(\pi_T)$, where $\pi_T\in\fS(132)$ is the permutation
corresponding 
to $T$ by the bijection above.
Moreover, let $X_{\gs,L}:=X_{\gs}(T_L)$ and $X_{\gs,R}:=X_{\gs}(T_R)$, where
$L$ and $R$ are the left and right subtrees of $T$.

We can translate the recursion \refL{LR} to 
recursive relations for the variables $X_\gs=X_\gs(T)$ as follows.
(We usually omit the argument $T$ for notational convenience.)

\begin{lemma}
  \label{LRT}
Let\/ $\gs\in\fS_k(132)$ with $k\ge1$ and
define $m$ and $\gD$ as in \refL{LR}.
Then, 
for any binary tree $T$,
\begin{multline}\label{lrt}
X_{\gs}
=
X_{\gs,L}+X_{\gs,R}
+\sum_{q\in \gD}
 X_{\gsxx1q,L}X_{\gsxx{q+1}k,R}
\\
+X_{\gsxx1{m-1},L}X_{\gsxx{m+1}k,R}.
 \end{multline}
\qed
\end{lemma}

Note also that 
$X_\gs=0$ unless $\gs\in\fSxzzz$ and, by \eqref{11},
\begin{equation}\label{11t}
  \sum_{\gs\in\fS_k} X_\gs = \binom nk.
\end{equation}

As an illustration and for later use, we write the recursion \eqref{lrt}
explicitly for some small $\gs$. 
For (notational) convenience, we  define
 $N=N(T):=X_1(T)=|T|$
and $Y=Y(T):=X_{12}(T)$, and define $N_L, N_R,
Y_L, Y_R$ correspondingly.
Note that then,
by \eqref{11t}, 
\begin{equation}\label{x21}
X_{21}=\binom N2-X_{12} =\binom N2-Y  .
\end{equation}

\begin{example}\label{Erec3} 
Taking $\gs=1$, 12, 123, 213, 231, 312
in \refL{LRT} we find the following recursions, 
noting that in these cases $\gD=\gD_\gs$ is empty except $\gD_{231}=\set2$ and 
$\gD_{312}=\set1$;
for \eqref{r213} we use also \eqref{x21}.
\begin{align}
N&=N_L+N_R+1, \label{rn}
\\
  Y&=Y_L+Y_R+N_L,  \label{ry}
\\
X_{123}&=X_{123,L}+X_{123,R}+Y_L,  \label{r123}
\\
X_{213}&=X_{213,L}+X_{213,R}+\binom{N_L}2-Y_L,  \label{r213}
\\
X_{231}&=X_{231,L}+X_{231,R}+Y_LN_R+N_LN_R, \label{r231}
\\
X_{312}&=X_{312,L}+X_{312,R}+N_LY_R+Y_R. \label{r312}
\end{align}
(These recursions can also easily be verified directly, and
\eqref{rn} is utterly trivial.)
\end{example}

Let $T_n$ be a uniformly random binary tree in $\cB_n$.
Note that $T_n$ by the bijection above corresponds to  a uniformly random
permutation in $\fS_n(132)$, \ie{} we can identify $T_n=T(\pinzzz)$.
With this identification and the notations above we have
\begin{equation}\label{X=n}
  X_\gs(T_n)=\ns(\pinzzz);
\end{equation}
we will in the sequel use this without comment and
study the random variables $X_\gs(T_n)$ when proving \refT{Tmain}.

\begin{remark}\label{Rbij}
  The bijection with $\cB_n$ is equivalent to a bijection with the set of
  Dyck paths of length $2n$, by the well-known standard bijection between
  the latter and $\cB_n$. This is equivalent to the bijection by
\citet[Exercises 2.2.1-3,5]{KnuthI} between $312$-avoiding permutations
and Dyck paths.
Another bijection with Dyck paths is given by
\cite{Krattenthaler}. For similar bijections of $\fS_n(123)$ and Dyck paths,
see \eg{} \cite{BJS}, \cite{Krattenthaler}, \cite{CEF}.
See also the many bijections with various objects in 
\citet[Exercise 6.19 (and its solution)]{StanleyII}. 
\end{remark}

\section{Expectations}\label{Sexp} 

We next  use an idea from \cite{SJ146} and 
consider the functionals 
$X_\gs$ above for 
another random binary tree $\td$ defined as follows, for $0<\gd<1$.
Note that this random tree, unlike $\tn$, has a random size.

We start with the root; we then add each of the two possible
children of the root with probability $p:=(1-\gd)/2$ each, and we continue
in the same way with the possible children of any node that we add to the tree,
with all random choices independent. Thus
$\td$ is a random Galton--Watson tree with offspring
distribution $\Bi(2,p)$. Since this offspring distribution has
expectation $2p=1-\gd<1$, the Galton--Watson tree $\td$ is subcritical and thus
\as{} finite.

The construction implies that if  $T=\td$, then the subtrees $T_L$ and $T_R$
are independent random trees; furthermore,
each of them empty with probability
$1-p=(1+\gd)/2$ and otherwise it has the same distribution as $T$.
(This can be used as an alternative, recursive definition of $\td$.)

\begin{remark}
  The argument in \cite{SJ146} uses full binary trees, which makes the
  details a little different although the main idea is the same. 
We thus present the argument in detail below,
  and refer the interested reader to \cite{SJ146} for comparisons.
\end{remark}

We let $\ed$ denote expectation of random variables defined for the random
tree 
$T=\td$. 
These expectations are generating functions in disguise.
In fact, let $Z=Z(T)$ be an arbitrary functional such that $|Z(T)|\le
C|T|^m$ for some constants $C$ and $m$. 
(This guarantees that all expectations and sums below converge, and is
satisfied by the functionals that we consider, \viz{} $X_\gs$ and products
of these.)
We write $z_n:=\E Z(T_n)$.  

\begin{lemma}\label{LedZ}
Let $Z$ and $z_n:=\E Z(T_n)$ be as above.  Then
\begin{equation}\label{sw2}
  \begin{split}
\ed Z 
&= \frac{1+\gd}{1-\gd} \sumni
z_n C_n  \parfrac{1-\gd^2}{4}^n.
  \end{split}
\end{equation}
\end{lemma}
\begin{proof}
  There are $C_n=\binom{2n}n/(n+1)$ trees in $\cB_n$.
If $T\in\cB_n$, then $T$ has $n$ nodes, with 2
  potential children each. Of these $2n$ potential children, $n-1$ exist and
  $n+1$ do not exist. The probability that $\td$ equals a given tree
  $T\in\cB_n$ is thus 
  \begin{equation}\label{sww}
\PP(\td = T) = p^{n-1}(1-p)^{n+1}= 2^{-2n}(1-\gd)^{n-1}(1+\gd)^{n+1}.	
  \end{equation}
This probability is the same for all $T\in \cB_n$, 
and since $|\cB_n|=C_n$, it follows that
the probability that $\td$ has order $n$ is
\begin{equation}\label{sw1}
\PP(|\td|=n)= \PP(\td\in\cB_n)= C_n 2^{-2n}(1-\gd)^{n-1}(1+\gd)^{n+1}
= C_n \frac{1+\gd}{1-\gd} \parfrac{1-\gd^2}{4}^n.
\end{equation}
Moreover, 
since \eqref{sww} does not depend on the choice of $T\in\cB_n$, we see that
conditioned on $|\td|=n$, $\td$ is uniformly distributed in
$\cB_n$;
in other words $\bigpar{\td\mid|\td|=n}\eqd T_n$.
Hence,
$\E \bigpar{Z\mid |\td|=n} =\E Z(T_n)=z_n$ and, using \eqref{sw1},
\begin{equation*}
  \begin{split}
\ed Z& = \sumni \PP(|\td|=n)\E (Z\mid |\td|=n) 	
 = \sumni \PP(|\td|=n) z_n
\\
&=\sumni
z_n C_n \frac{1+\gd}{1-\gd} \Bigparfrac{1-\gd^2}{4}^n.
\qedhere
  \end{split}
\end{equation*}
\end{proof}

By \refL{LedZ}, $\ed Z$ is, apart from the factor $(1+\gd)/(1-\gd)$, the ordinary
generating function of the sequence $C_nz_n$, evaluated at $(1-\gd^2)/4$.
Conversely, by taking $\gd=\sqrt{1-4x}$ in \eqref{sw2}, we obtain,
for $0<x<1/4$,
\begin{equation}\label{sw3}
\sumni C_nz_nx^n = \frac{1-\sqrt{1-4x}}{1+\sqrt{1-4x}} \E_{\sqrt{1-4x}} Z
=\frac{1-2x-\sqrt{1-4x}}{2x} \E_{\sqrt{1-4x}} Z.
\end{equation}
Note that $Z=1$ yields the well-known \gfx{} for the Catalan numbers, see
\eg{} \cite[p.~35]{FS}.

\begin{remark}\label{Rrational}
 For the variables $Z$ that we study below (products of $X_\gs$), 
$\ed Z$ turns out to be a polynomial
 in $\gd\qw$; in this case \eqref{sw3} yields the
 generating function $\sumni C_nz_nx^n$ as a rational function of
 $\sqrt{1-4x}$. By analytic continuation, the resulting formula is valid for
 all complex $x$ with $|x|<1/4$, and the generating function extends to an
 analytic function in
 $\bbC\setminus[1/4,\infty)$.
\end{remark}

We can now apply  singularity analysis and obtain asymptotics of $z_n$ from
asymptotics of $\ed Z $ as $\gd\downto0$. (Note that although we can define
the random tree $T_\gd$ for $\gd=0$, 
which will be a critical Galton--Watson
tree and thus \as{} finite,  the expectations that we are interested will
all be infinite and of no use to us; hence we consider $\gd>0$ and take
asymptotics.) 
We state a simple case that is enough for our purposes.
We let in this section (and the next)
$O(\gd^{-m})$ denote an arbitrary \emph{polynomial in
$\gd\qw$} of degree at most $m$. 

\begin{lemma}
\label{Lsing}
If\/ $\ed Z=a\gdx{m}+\ogdx{(m-1)}$, where $m\ge1$ and $a\neq0$, then
\begin{equation*}
\E Z(T_n)\sim a\frac{\Gamma(1/2)}{\Gamma(m/2)}n^{(m+1)/2}
\qquad\text{as }\ntoo.
\end{equation*}
\end{lemma}

\begin{proof}
By \refR{Rrational}, the
 generating function $\sumni C_nz_nx^n$ extends to an analytic function in
 $\bbC\setminus[1/4,\infty)$, and as $x\to1/4$, 
by assumption and \eqref{sw3}, 
\begin{equation*}
\sumni c_{n}z_{n}x^{n}\sim
a\frac{1-\sqrt{1-4x}}{1+\sqrt{1-4x}} (1-4x)^{-m/2}
\sim a(1-4x)^{-m/2}.
\end{equation*}
This implies by standard singularity analysis
(see \cite[Corollary VI.1]{FS}), 
\begin{equation*}
 c_{n}z_{n} \sim a 4^n \frac{n^{m/2-1}}{\gG(m/2)}.
\end{equation*}
The result follows by  this and the standard asymptotic expression 
$C_n\sim 4^n/\sqrt{\pi n^3}$ for the Catalan numbers \cite[p.~38]{FS}.
\end{proof}

For later use, we show also the following, recalling $N(T):=|T|$.

\begin{lemma}
\label{L:Z5}
\begin{thmenumerate}
\item 
Let $f(\gd)=\ed Z$. Then
\begin{equation}\label{lz5}
  \ed(NZ) = -\frac12\bigpar{{\gd\qw}-\gd}f'(\gd)+\gd\qw f(\gd).
\end{equation}
\item 
In particular,
if\/ $\ed Z=a\gdx{m}+\ogdx{(m-1)}$, 
where $m\ge1$ and $a\in\bbR$,
then
$\ed (NZ)=\tfrac12ma\gdx{(m+2)}+\ogdx{(m+1)}$.
\end{thmenumerate}
\end{lemma}

\begin{proof}
\pfitem{i}
Differentiate \eqref{sw2}. This gives, using \eqref{sw2} also for $NZ$,
  \begin{equation*}
  \begin{split}
\frac{\dd}{\dd\gd}\ed Z 
&= \frac1{1+\gd}\ed Z + \frac{1}{1-\gd} \ed Z
+\frac{1+\gd}{1-\gd} \sumni
z_n C_n \frac{-2\gd n}{1-\gd^2} \Bigparfrac{1-\gd^2}{4}^n
\\&
=\frac{2}{1-\gd^2}\ed Z -\frac{2\gd}{1-\gd^2}\ed(NZ).
  \end{split}
  \end{equation*}
The formula \eqref{lz5} follows. 

\pfitem{ii} An immediate consequence of \eqref{lz5}.
\end{proof}

As an example, taking $Z=1$ yields $f(\gd)=1$, and thus \eqref{lz5} yields
\begin{equation}
  \ed N = \gd\qw. \label{edn}
\end{equation}
Taking $Z=N$ in \eqref{lz5} now yields 
\begin{equation}\label{edn2}
  \ed N^2 = \frac12\gdx3 +\gdx2 -\frac12\gdx1,
\end{equation}
and we can continue and find explicit expressions for $\ed N^m$ for any
desired $m$.
(One can check that \refL{Lsing} is correct but trivial in these cases.)

After these preliminaries, we now consider the variables $X_\gs$, and begin
with their expectations for $\td$. 
Recall that $\gl(\gs)$ is defined by \eqref{gl}.  

\begin{lemma}\label{Led}
Let\/ $\gs\in\fS_k(132)$ with $k=|\gs|\ge1$ and
define $m$ and $\gD$ as in \refL{LR}.
Then\/ $\ed X_\gs$ is a polynomial in $\gdx1$ of
  degree $\gl(\gs)-1$ 
given by the recursion $\ed X_1=\gdx1$ and, for $k>1$,
\begin{multline*}
\ed X_{\gs}
=
\gd\qw\frac{(1-\gd)^2}4 
\sum_{q\in \gD}
\ed X_{\gsxx1q}\ed X_{\gsxx{q+1}k}
\\
+
\begin{cases}
\frac12(\gd\qw-1) \ed  X_{\gsxx{2}k},
& m=1, \\
\frac14 \gd\qw(1-\gd)^2 \ed  X_{\gsxx1{m-1}}\ed X_{\gsxx{m+1}k},
& 1<m<k, \\
\frac12(\gd\qw-1) \ed  X_{\gsxx1{k-1}},
& m=k.
\end{cases}
 \end{multline*}
The polynomial $\ed X_\gs$ has  
leading term $\eegs\gdxx{\gl(\gs)-1}$ 
and vanishing constant term,
where $\eegs>0$ satisfies the recursion $\ee_1=1$
and,
for $k>1$,
\begin{equation}\label{eegs}
\ee_{\gs}
=
\frac14\sum_{q\in \gD}
 \ee_{\gsxx1q}\ee_{\gsxx{q+1}k}
+ \begin{cases}
 \frac12\ee_{\gsxx1{k-1}}, & m=k,
\\
0, & m<k.  
 \end{cases}
\end{equation}

\end{lemma}

\begin{proof}
  We use induction on $\gl(\gs)$. 
We use the recursion in \refL{LRT} and take expectations, considering the
terms on the \rhs{} of \eqref{lrt} separately.

Since $T_L$ is a copy of $T=\td$ with probability $p=(1-\gd)/2$ and empty
with probability $1-p=(1-\gd)/2$, 
and the same holds for $T_R$,
we have
\begin{equation}
  \ed X_{\gs,L}=\ed X_{\gs,R}=p\ed X_\gs = \frac{1-\gd}2\ed X_\gs.
\end{equation}

Furthermore, $T_L$ and $T_R$ are independent, and thus, for $q\in\gD$,
\begin{equation}\label{visby}
  \begin{split}
\ed\bigpar{X_{\gsxx1q,L}X_{\gsxx{q+1}k,R}}	
&=
\ed\bigpar{X_{\gsxx1q,L}}\ed\bigpar{X_{\gsxx{q+1}k,R}}	
\\&
= \Bigparfrac{1-\gd}{2}^2
\ed\bigpar{X_{\gsxx1q}}\ed\bigpar{X_{\gsxx{q+1}k}}.
  \end{split}
\end{equation}
By the induction hypothesis, this is a polynomial in $\gdx1$ of degree 
\begin{equation}
  \begin{split}
&  \gl(\gsxx1q)-1+\gl(\gsxx{q+1}k)-1
\\&\qquad
=q+D(\gsxx1q)-1+k-q+D(\gsxx{q+1}k)-1
\\&\qquad
=k+D(\gsxx1k)-2
=\gl(\gs)-2,	
  \end{split}
\end{equation}
recalling that $q\in\gD$ implies that $q$ is a descent in $\gs$, which implies
$D(\gsxx1q)+D(\gsxx{q+1}k)=D(\gsxx1k)$ by our definition of $D$.
(Note that the induction assumption that the expectations are polynomials 
with vanishing constant term is used to
guarantee that the right hand side of \eqref{visby} is a polynomial in
$\gdx1$, even though it
contains the factor $(1-\gd)^2$; the same applies below.)

For the final term in \eqref{lrt}, we consider four different cases.
First, if $1<m<k$, then  as in \eqref{visby}
\begin{equation}
  \begin{split}
\ed\bigpar{X_{\gsxx1{m-1},L}X_{\gsxx{m+1}k,R}}	
&=
\Bigparfrac{1-\gd}{2}^2
\ed\bigpar{X_{\gsxx1{m-1}}}\ed\bigpar{X_{\gsxx{m+1}k}},
  \end{split}
\end{equation}
and this is a polynomial in $\gdx1$ of degree
\begin{equation}
  \begin{split}
&  \gl(\gsxx1{m-1})-1+\gl(\gsxx{m+1}k)-1
\\&\qquad
=m-1+D(\gsxx1{m-1})-1+k-m+D(\gsxx{m+1}k)-1
\\&\qquad
=\gl(\gs)-3.	
  \end{split}
\end{equation}

If $m=1<k$, then the final term of \eqref{lrt}
is simply $X_{\gsxx2k,R}$, with an expectation that by induction is
a polynomial in $\gdx1$ of degree
\begin{equation}
  \begin{split}
& \gl(\gsxx{2}k)-1
=k-1+D(\gsxx2{k})-1
=\gl(\gs)-3,
  \end{split}
\end{equation}
since 1 is a descent.

If $m=k>1$, then the final term of \eqref{lrt}
is similarly $X_{\gsxx1{k-1},L}$, with an expectation that by induction is
a polynomial in $\gdx1$ of degree
\begin{equation}
  \begin{split}
& \gl(\gsxx{1}{k-1})-1
=k-1+D(\gsxx1{k-1})-1
=\gl(\gs)-2,
  \end{split}
\end{equation}
since $k-1$ is not a descent in $\gs$.

Finally, if $m=k=1$, \ie, if $\gs=1$, the final term is simply 1, again a
polynomial of degree $\gl(\gs)-2$.

Collecting the terms above, we thus obtain from \eqref{lrt}
\begin{equation}
  \label{ej3}
\ed X_\gs =  2p \ed X_\gs + f(\gd)
= (1-\gd) \ed X_\gs + f(\gd),
\end{equation}
where $f(\gd)$ is shorthand for a polynomial in $\gdx1$ of degree (at most)
$\gl(\gs)-2$,
which yields
\begin{equation}
  \label{ej4}
\ed X_\gs=\gd\qw f(\gd),
\end{equation}
a polynomial in $\gdx1$ of degree (at most)
$\gl(\gs)-1$ and without constant term.
Writing $f(\gd)$ explicitly, this yields the recursion stated in the lemma.
For $\gs=1$ we have $f(\gd)=1$ and \eqref{ej4} yields $\ed X_1=\gd\qw$, as
was found in 
another way in \eqref{edn}.

Moreover, an inspection of the leading terms above shows that the leading
coefficient of $f(\gd)$ is $\eegs$ given by \eqref{eegs} when $|\gs|>1$,
and $\ee_1=1$.
Thus, by induction, $\eegs>0$. (Recall that $\gD\neq\emptyset$ if $m<k$, so the
\rhs{} of \eqref{eegs} contains at least one non-zero term.)

This completes the induction step.
\end{proof}

It is now easy to show \eqref{tmaine}.

\begin{corollary}\label{Ced}
For every $\gs\in\fSxzzz$,
  \begin{equation}\label{tmaine=}
  \E \ns(\pinzzz) = \E X_\gs(T_n)
\sim A_\gs n^{\gl(\gs)/2},
\end{equation}
where
\begin{equation}\label{ags}
  A_\gs = \frac{\sqrt\pi}{\gG((\gl(\gs)-1)/2)} \eegs,
\end{equation}
with $\eegs$ given by the recursion \eqref{eegs}.
\end{corollary}
\begin{proof}
  Immediate from Lemmas \ref{Led} and \ref{Lsing}, together with \eqref{X=n}.
\end{proof}

\begin{example}
\label{Eed3}
For $|\gs|=1$, we have $\ed N=\ed X_1=\gdx1$, as stated in \eqref{edn}.

For $|\gs|=2$, we have two cases. For $X_{12}=Y$ we obtain, \cf{}
\eqref{ry},
\begin{equation}
  \ed X_{12} = \txfrac12(\gdx1-1)\ed X_1 
= \txfrac12\gd\qww-\txfrac12\gd\qw, \label{edy}
\end{equation}
Similarly, by \refL{Led} (with $\gD=\set1$) and a short calculation,
or by \eqref{x21}, \eqref{edn}--\eqref{edn2} and \eqref{edy},
\begin{equation}
  \ed X_{21} 
= \txfrac14\gdx3-\txfrac14\gd\qw, \label{ed21}
\end{equation}

For $|\gs|=3$, we obtain from \refL{Led}, or similarly from the explicit
recursions in \refE{Erec3}, by simple calculations,
\begin{align}
\ed X_{123} &= \txfrac14\gdx3 - \txfrac12\gdx2 + \txfrac14\gdx1 \label{ed123},
\\
\ed X_{213} &= \txfrac18\gdx4-\txfrac18\gdx3 - \txfrac18\gdx2 + \txfrac18\gdx1,
\label{ed213}
\\
\ed X_{231} &= \txfrac18\gdx4-\txfrac18\gdx3 - \txfrac18\gdx2 + \txfrac18\gdx1,
\label{ed231}
\\
\ed X_{312} &= \txfrac18\gdx4-\txfrac18\gdx3 - \txfrac18\gdx2 + \txfrac18\gdx1,
\label{ed312}
\\
\ed X_{321} &= \txfrac18\gdx5-\txfrac18\gdx4 - \txfrac18\gdx3 + \txfrac18\gdx2.
\label{ed321}
\end{align}
Note that $\ed X_{213}=\ed X_{231}=\ed X_{312}$, which by \refL{LedZ} is
equivalent to the result by \citet{Bona-surprising}
$  \E n_{213}(\pinx{132})=  \E n_{231}(\pinx{132})=  
\E n_{312}(\pinx{132})$,
as mentioned earlier in \eqref{e213}.

The asymptotics \eqref{e12}--\eqref{e321} follow from  \refC{Ced} and
\eqref{eegs}. 
Alternatively, we can obtain these from the explicit
formulas \eqref{edy}--\eqref{ed321} and \refL{Lsing}.
\end{example}

\begin{remark}
  When $Z=X_\gs=\ns(\pinzzz)$, 
$z_n$ is the expected
  number of occurrences of $\gs$ in a random permutation in $\fSnzzz$,
  and $C_nz_n$ is thus the total number of occurrences of $\gs$ in all
  permutations in $\fSnzzz$. 
Generating functions for the latter numbers
have been given for the cases in \refE{Eed3} 
(although not explicitly for 321)
by \citet{Bona-abscence} and
  \cite{Bona-surprising}; 
by \refL{LedZ} and \refR{Rrational}, the formulas \eqref{edy}--\eqref{ed312} 
are equivalent to his results.
\end{remark}

\begin{remark}
  As said in \refS{S:intro}, $n_{231}(\pinzzz)$ and $n_{312}(\pinzzz)$ have
  the same distribution by symmetry, and thus $\ed X_{231}=\ed X_{312}$ is
  obvious. It is interesting that the proof above obtains these 
coinciding expectations by different routes, using the different recursions
\eqref{r231} and \eqref{r312}. The same applies to the higher moments
treated below:
$\ed X_{231}^k=\ed X_{312}^k$ for any $k$, but that is difficult to see from
our recursions.
\end{remark}

\begin{example}
  \label{Eed4}
For $|\gd|=4$, there are $C_4=14$ permutations $\gs\in\fS_4$.
\refL{Led} yields the following formulas.
\begin{align}\label{ed1234}
 \ed X_{1234} &= 
\txfrac18{\gd}^{-4}-\txfrac38{\gd}^{-3}+\txfrac38{\gd}^{-2}-\txfrac18\gd^{-1}
\\
\ed X_{2134} &
=\ed X_{2314} 
=\ed X_{2341} 
=\ed X_{3124} 
=\ed X_{3412} 
=\ed X_{4123} 
\notag\\&=
\txfrac{1}{16}{\gd}^{-5}-\txfrac18{\gd}^{-4}+\txfrac18{\gd}^{-2}
 -\txfrac{1}{16}\gd^{-1}
\\
\ed X_{3214} &
=\ed X_{3421}
=\ed X_{4231}
=\ed X_{4312}
\notag\\&
= 
\txfrac{1}{16}{\gd}^{-6}-\txfrac18{\gd}^{-5}+\txfrac18{\gd}^{-3}
 -\txfrac{1}{16}{\gd}^{-2}
\\
\ed X_{3241}&
=\ed X_{4213}
\notag\\&
=
\txfrac{1}{32}{\gd}^{-6}-\txfrac{1}{32}{\gd}^{-5}-\txfrac{1}{16}{\gd}^{-4}
+\txfrac{1}{16}{\gd}^{-3}+\txfrac{1}{32}{\gd}^{-2}-\txfrac{1}{32}\gd^{-1}
\\ \label{ed4321}
\ed X_{4321}
&={\txfrac {5}{64}}{\gd}^{-7}-{\txfrac {5}{32}}{\gd}^{-6}
-{\txfrac {1}{64}}{\gd}^{-5}+\txfrac{3}{16}{\gd}^{-4}
-{\txfrac {5}{64}}{\gd}^{-3}-\txfrac{1}{32}{\gd}^{-2}+{\txfrac {1}{64}}\gd^{-1}.
\end{align}
We see again several coincidences, which by \refL{LedZ} imply corresponding
equalities for each $n$, for a random 132-avoiding permutation $\pinzzz$:
\begin{align}
\E n_{2134} &
=\E n_{2314} 
=\E n_{2341} 
=\E n_{3124} 
=\E n_{3412} 
=\E n_{4123},
\\
\E n_{3214} &
=\E n_{3421}
=\E n_{4231}
=\E n_{4312},
\\
\E n_{3241}&
=\E n_{4213}.
\end{align}
Some equalities are obvious by the inversion symmetry in \refR{Rsymm},
others follow by \citet{Bona-surprising} and all are contained in the
result by \citet{Rudolph}.

Asymptotic results follow by 
\eqref{ed1234}--\eqref{ed4321} and \refL{Lsing}, or directly by \refC{Ced}
and \eqref{eegs}; we leave these to the reader.
It is also possible to obtain exact expressions for finite $n$
by \eqref{sw3} and Taylor
expansion; we leave these too to the reader.
\end{example}

\begin{example}
  \label{Eed1...k}
For $\gs=1\dotsm k$, $k\ge1$, we have $\gD=\emptyset$ and \refL{Led} yields by
induction in $k$
\begin{equation}\label{ed1...k}
  \ed X_{1\dotsm k} = 2^{1-k}\bigpar{\gdx1-1}^{k-1}\gdx1.
\end{equation}
This is by \refL{LedZ} and \refR{Rrational} equivalent to the generating
function given for this case by \citet{Bona-abscence}.

\refL{Lsing} and \eqref{ed1...k} yield \eqref{e1...k}.
\end{example}

\begin{example}
  \label{Eedk...1}  
For $\gs=k\dotsm 1$, $k\ge1$, we have the opposite extreme 
$\gD=[k-1]$. \refL{Led} yields the recursion, where we 
write $f_k(\gd)=\ed X_{\gski}$,
\begin{equation}\label{edk...1}
  f_k(\gd) = \tfrac14\gdx1(1-\gd)^2\sum_{q=1}^{k-1}f_q(\gd)f_{k-q}(\gd)
+ \tfrac12(\gdx1-1)f_{k-1}(\gd),
\end{equation}
which by \eqref{sw3} is equivalent to the recursion 
given for the corresponding generating functions in \citet{Bona-abscence}.

The leading term $B_{\gski}\gd^{-(2k-1)}$
is given by the recursion \eqref{eegs}, but it is simpler
to argue backwards and note that $A_{\gski}=1/k!$ by \refT{Tmain},
see also \refR{Rk1}, and thus  \eqref{ags} yields
\begin{equation}
  B_{\gski} = \frac{\gG(k-1/2)}{\gG(1/2) k!} 
=\frac{(2k-3)!!}{2^{k-1}k!}
=\frac{C_{k-1}}{2^{2k-2}}.
\end{equation}
See the examples in \eqref{ed21}, \eqref{ed321}, \eqref{ed4321}.
\end{example}

\section{Higher moments}\label{Shigher}

We can compute higher moments in the same way. 

\begin{lemma}\label{Lmom1}
For any permutations
$\gspp{1},\dots,\gspp{\nu}\in\fSx(132)$,
not necessarily distinct,
$\ed \bigpar{X_{\gspp1}\dotsm X_{\gspp{\nu}}}$ 
is a polynomial in $\gdx1$ of
degree $\sum_{j=1}^\nu\gl(\gspp{j})-1$,
with positive leading coefficient 
$\ee_{\gspp1,\dots,\gspp\nu}$ 
and vanishing constant term. 
\end{lemma}

\begin{proof}
  We argue as in the proof of \refL{Led}, using induction on 
$\sum_{j=1}^\nu\gl(\gspp{j})$.
Replace each $X_{\gspp j}$ by the corresponding expression in
\eqref{lrt}, expand the product of these, and take the expectation.
Among the 
many terms that this produces, the two special ones
$\ed \bigpar{X_{\gspp1,L}\dotsm X_{\gspp{\nu},L}}$ and
$\ed \bigpar{X_{\gspp1,R}\dotsm X_{\gspp{\nu},R}}$ 
are both equal to 
$p \ed \bigpar{X_{\gspp1}\dotsm X_{\gspp{\nu}}}$.
All other terms are by induction polynomials in $\gdx1$, of degrees at most  
$\sum_{j=1}^\nu\gl(\gspp{j})-2$ 
(by arguing similarly to the proof of \refL{Led} for each $\gspp j$);
moreover, 
there is at least one term of exactly this degree and
all polynomials have positive leading coefficients.
The result follows as in \refL{Led}.
\end{proof}

\begin{example}
By squaring
\eqref{ry} and taking the expectation we obtain
\begin{equation*}
  \begin{split}
\ed Y^2 
&= 	
\ed (Y_L+N_L)^2 + \ed Y_{R}^2 + 2 \ed (Y_{L}+N_L)\ed Y_R
\\
&= 
p\ed (Y+N)^2 +p \ed Y^2 + 2p^2 \ed (Y+N)\ed Y
\\&
=2p\ed Y^2+2p\ed(NY)+p\ed N^2 + 2p^2 (\ed Y)^2
+2p^2\ed Y\ed N.
  \end{split}
\end{equation*}
Hence, recalling $2p=1-\gd$,
\begin{equation*}
  \begin{split}
\ed Y^2 
&= 	
\gdx1\Bigpar{2p\ed(NY)+p\ed N^2 + 2p^2 (\ed Y)^2
+2p^2\ed Y\ed N},
  \end{split}
\end{equation*}
which can be written as an explicit polynomial in $\gdx1$ by 
\eqref{edn}--\eqref{edn2}, \eqref{edy} and
\refL{L:Z5}. Using this, we then can find, for example, $\E (X_{123}Y)$ 
by multiplying \eqref{ry} and \eqref{r123} and taking the expectation, and
then
$\E X_{123}^2$ by squaring \eqref{r123} and using the same argument
again. 
\end{example}

In this way we can recursively
obtain any mixed moment of the variables $X_\gs$
as a polynomial in $\gdx1$. 
For simplicity, we leave  exact formulas to the reader, and consider
only the leading terms, which 
by \refL{Lsing}
will yield the moment asymptotics for $T_n$
that we desire. 

A recursion for the leading coefficients $\ee_{\gspp1,\dots,\gspp\nu}$ 
is implicit in the proof above, but to write it explicitly in
general seems a bit messy, so we restrict ourselves in the rest of this
section to the case $|\gs|\le3$, which gives examples illustrating the
general behaviour.

We consider first a single $X_{\gs}$ with $|\gs|=3$,
but for the induction, we have to consider mixed moments of $X_\gs$ and
$Y=X_{12}$. 

\begin{lemma}\label{LedXY}\quad
  \begin{romenumerate}[-10pt]
  \item \label{ledxy123}
  If $k\ge0$ and $l\ge0$ with $k+l\ge1$, then
  \begin{equation}\label{xy123}
\ed\bigpar{X_{123}^kY^l}
=
a_{kl}\gdxx{4k+3l-1} + \ogdxx{4k+3l-2}
  \end{equation}
for some positive numbers $a_{kl}$ satisfying $a_{01}=\frac12$, $a_{10}=\frac14$
and the recursion relation
\begin{equation}\label{xy123r}
a_{k,l}
=\frac{k}2 a_{k-1,l+1} 
+\frac{l(4k+3l-4)}4 a_{k,l-1}
+
\frac{1}4\sumsum_{0<i+j<k+l} 
\binom{k}{i}\binom{l}{j}
 a_{i,j}a_{k-i,l-j}.
\end{equation}

\item \label{ledxy213}
  If $k\ge0$ and $l\ge0$ with $k+l\ge1$, then
  \begin{equation}\label{xy213}
\ed\bigpar{X_{213}^kY^l}
=
b_{kl}\gdxx{5k+3l-1} + \ogdxx{5k+3l-2}
  \end{equation}
for some positive numbers $b_{kl}$ satisfying $b_{01}=\frac12$, $b_{10}=\frac18$
and the recursion relation
\begin{multline}\label{xy213r}
b_{k,l}
=\frac{k(5k+3l-6)(5k+3l-4)}{16} b_{k-1,l} 
+\frac{l(5k+3l-4)}4 b_{k,l-1}
\\
+
\frac{1}4\sumsum_{0<i+j<k+l} 
\binom{k}{i}\binom{l}{j}
 b_{i,j}b_{k-i,l-j}.
\end{multline}

\item \label{ledxy231}
  If $k\ge0$ and $l\ge0$ with $k+l\ge1$, then
  \begin{equation}\label{xy231}
\ed\bigpar{X_{231}^kY^l}
= \ed\bigpar{X_{312}^kY^l}
=
c_{kl}\gdxx{5k+3l-1} + \ogdxx{5k+3l-2}
  \end{equation}
for some positive numbers $c_{kl}$ satisfying $c_{01}=\frac12$, $c_{10}=\frac18$
and the recursion relation
\begin{multline}\label{xy231r}
c_{k,l}
=\frac{l(5k+3l-4)}4 c_{k,l-1}
+
\frac{1}4\sumsumsum_{(i,j,m)\neq(0,0,0),(k,l,0)} 
\binom{k}{i,m,k-i-m}\binom{l}{j} \times\\
\frac{\gG\bigpar{(5i+3j-1)/2+m)}}{\gG\bigpar{(5i+3j-1)/2)}}
 c_{i,j}c_{k-i-m,l-j+m}.
\end{multline}
  \end{romenumerate}
\end{lemma}

\begin{proof}
\pfitemref{ledxy123}
Note that \eqref{edy} and \eqref{ed123} show that \eqref{xy123} holds when
$k+l=1$, with 
$a_{01}=\frac12$,
$a_{10}=\frac14$. 
We continue by induction, and assume that $K,L\ge0$ with
$K+L\ge2$ 
are such that \eqref{xy123} holds 
when $1\le 4k+3l<4K+3L$.
For such $k$ and $l$ and any $m\ge0$, 
\refL{L:Z5}(ii) implies
\begin{equation}\label{z10}
\begin{split}
\ed(X_{123}^kY^lN^m)
&=a_{kl}\prod_{j=0}^{m-1}\frac{4k+3l+2j-1}2\cdot\gdx{(4k+3l+2m-1)}
\\&\hskip12em
+\ogdx{(4k+3l+2m-2)}\\
&=\ogdx{(4k+3l+2m-1)}.
\end{split}
\raisetag{\baselineskip}
\end{equation}
The same holds for $k=l=0$ and $m\ge1$ too (with $a_{00}=-2$) by
\eqref{edn} and Lemma \ref{L:Z5}.

Now consider $k=K$ and $l=L$. 
By \eqref{r123}, \eqref{ry} and the binomial theorem,
\begin{multline}\label{kl123}
\ed\bigpar{X_{123}^kY^l}
=\sum_{k_1+k_2+k_3=k}\sum_{l_1+l_2+l_3=l} 
\binom{k}{k_1,k_2,k_3}\binom{l}{l_1,l_2,l_3}
\times\\	
\ed\bigpar{X_{123,L}^{k_1}X_{123,R}^{k_2}Y_L^{k_3+l_1}Y_R^{l_2}N_L^{l_3}}.
\end{multline}
Consider one of the terms in the sum.
If this term contains  both $L$-factors and $R$-factors,
\ie, if $k_1+k_3+l_1+l_3>0$ and $k_2+l_2>0$,
then the expectation is, by the induction hypothesis and \eqref{z10},
\begin{equation}
\begin{split}
p^2  \ed\bigpar{X_{123}^{k_1}Y^{k_3+l_1}N^{l_3}}
\ed\bigpar{X_{123}^{k_2}Y^{l_2}}
= \ogdxx{4k_1+3k_3+3l_1+2l_3+4k_2+3l_2-2}
\end{split}  
\end{equation}
If $k_3>0$ or $l_3>0$, this term is of lower order than $\gdxx{4k+3l-2}$,
and we see, using the induction hypothesis again, that the sum of the terms in
\eqref{kl123} 
with both $L$-factors and $R$-factors is
\begin{equation*}
  \begin{split}
\sumsum_{0<k_1+l_1<k+l} 
\binom{k}{k_1}\binom{l}{l_1}
p^2 a_{k_1,l_1}a_{k-k_1,l-l_1}	
\gdxx{4k+3l-2}
+ \ogdxx{4k+3l-3}.
  \end{split}
\end{equation*}
The terms in \eqref{kl123} with only $L$-factors are the ones with
$k_2=l_2=0$.
The induction hypothesis and \eqref{z10}
now show that the term is of order
$\ogdx{4k_1+4k_3+3l_1+2l_3-1}$, and thus only terms with $k_3+l_3\le1$ are
significant. The sum of these terms is thus, using \eqref{z10},
{\multlinegap=0pt
\begin{multline*}
p \ed (X_{123}^kY^l)	
+pk\ed (X_{123}^{k-1}Y^{l+1})	
+pl\ed (X_{123}^{k}Y^{l-1}N)
+\ogdxx{4k+3l-3}	
\\
\shoveleft{\quad=
p \ed (X_{123}^kY^l)	
+pk a_{k-1,l+1}  \gdxx{4k+3l-2}		
}\\
+pla_{k,l-1}\frac{4k+3l-4}2\gdxx{4k+3l-2}	
+\ogdxx{4k+3l-3}.
\end{multline*}}%
Finally, the only term in \eqref{kl123} with only $R$-factors is 
\begin{equation*}
 \ed (X_{123,R}^kY_R^l)	
=
p \ed (X_{123}^kY^l).	
\end{equation*}
Using $p=(1-\gd)/2$, we thus obtain by collecting the terms in \eqref{kl123},
{\multlinegap=0pt
\begin{multline*}
\gd \ed (X_{123}^kY^l)	
=\frac12 k a_{k-1,l+1}  \gdxx{4k+3l-2}		
+\frac12la_{k,l-1}\frac{4k+3l-4}2\gdxx{4k+3l-2}	
\\+
\sumsum_{0<k_1+l_1<k+l} 
\binom{k}{k_1}\binom{l}{l_1}
\frac14 a_{k_1,l_1}a_{k-k_1,l-l_1}	
\gdxx{4k+3l-2}
+ \ogdxx{4k+3l-3},
 \end{multline*}}%
which completes the induction.

\pfitemref{ledxy213}
Similar, with $4k$ replaced by $5k$ and using \eqref{r213}; the main
difference is that the 
significant terms with only $L$-factors now are
$ \ed (X_{123,L}^kY_L^l)$, 	
$k\ed (X_{123,L}^{k-1}Y_L^{l}\binom{N_L}2)$ and
$l\ed (X_{123,L}^{k}Y_L^{l-1}N_L)$, where the first and third terms are as
above and the second is handled by the analogue of \eqref{z10}. 

\pfitemref{ledxy231}
The equality $\ed\bigpar{X_{231}^kY^l}= \ed\bigpar{X_{312}^kY^l}$ follows
from the inversion symmetry in \refR{Rsymm}, which 
implies that $(X_{231},Y)\eqd(X_{312},Y)$
by translating first to $T_n$ by \eqref{X=n}
and then to $\td$ by taking a random $n$.
For the recursion we can use any of \eqref{r231} and \eqref{r312}; the
leading terms will be the same.
The main difference in the induction is that (using \eqref{r231})
the significant terms with both $L$-factors and $R$-factors now 
are all terms
\begin{multline*}
\binom{k}{k_1,k_2,k_3}\binom{l}{l_1}
\ed\bigpar{X_{231,L}^{k_1}X_{231,R}^{k_2}N_R^{k_3}Y_L^{k_3+l_1}Y_R^{l_2}}
\\
=
\binom{k}{k_1,k_2,k_3}\binom{l}{l_1}
p^2\ed\bigpar{X_{231}^{k_1}Y^{k_3+l_1}}\ed\bigpar{X_{231}^{k_2}Y^{l_2}N^{k_3}},
\end{multline*}
except the terms with $k_1+k_3+l_1=0$ or $k_2+k_3+l_2=0$, which, using the
analogue of \eqref{z10}, leads to the recursion \eqref{xy231r}.
(We write $i=k_2$, $j=l_2$, $m=k_3$.)
\end{proof}

\begin{remark}\label{R00}
  The proof (or a direct inspection) shows that the recursions
  \eqref{xy123r}, \eqref{xy213r}, \eqref{xy231r} hold also for $k+l=1$,
  provided we define $a_{0,0}=b_{0,0}=c_{0,0}:=-2$.
\end{remark}

This yields the moment asymptotics. 

\begin{theorem}\label{Tmom3}
The following hold as \ntoo, for any integers
$k\ge0$ and $l\ge0$.
  \begin{romenumerate}[-10pt]
  \item \label{ledmxy123}
  \begin{equation}\label{mxy123}
\nxx{4k+3l} \E\bigpar{X_{123}(T_n)^kY(T_n)^l}
\to\frac{k!\,l!\,\sqrt\pi}{2^{4k+3l-2}\,\gxx{4k+3l-1}} 
\ga_{kl}
  \end{equation}
for some numbers $\ga_{kl}$ satisfying $\ga_{0,0}=-1/2$,
$\ga_{10}=\ga_{01}=1$
and the recursion relation
\begin{equation}\label{mxy123r}
\ga_{k,l}
=\xpar{l+1} \ga_{k-1,l+1} 
+2\xpar{4k+3l-4} \ga_{k,l-1}
+
\sumsum_{0<i+j<k+l} 
 \ga_{i,j}\ga_{k-i,l-j}.
\end{equation}

\item \label{ledmxy213}
  \begin{equation}\label{mxy213}
\nxx{5k+3l}\E\bigpar{X_{213}(T_n)^kY(T_n)^l}
\to\frac{k!\,l!\,\sqrt\pi}{2^{5k+3l-2}\,\gxx{5k+3l-1}} 
\gb_{kl}
  \end{equation}
for some numbers $\gb_{kl}$ satisfying 
$\gb_{0,0}=-1/2$, $\gb_{10}=\gb_{01}=1$
and the recursion relation
\begin{multline}\label{mxy213r}
\gb_{k,l}
=2{(5k+3l-6)(5k+3l-4)} \gb_{k-1,l} 
+2\xpar{5k+3l-4} \gb_{k,l-1}
\\
+
\sumsum_{0<i+j<k+l} 
 \gb_{i,j}\gb_{k-i,l-j}.
\end{multline}

\item \label{ledmxy231}
    \begin{equation}\label{mxy231}
	  \begin{split}
\nxx{5k+3l}\E\bigpar{X_{231}(T_n)^kY(T_n)^l}
&=
\nxx{5k+3l}\E\bigpar{X_{312}(T_n)^kY(T_n)^l}
\\&
\to\frac{k!\,l!\,\sqrt\pi}{2^{5k+3l-2}\,\gxx{5k+3l-1}} 
\gc_{kl}		
	  \end{split}
  \end{equation}
for some  numbers $\gc_{kl}$ satisfying 
$\gc_{0,0}=-1/2$, $\gc_{10}=\gc_{01}=1$
and the recursion relation
\begin{multline}\label{mxy231r}
\gc_{k,l}
=2\xpar{5k+3l-4} \gc_{k,l-1}
+
\sumsumsum_{(i,j,m)\neq(0,0,0),(k,l,0)} 
\\
2^{2m}
\frac{\gG\bigpar{(5i+3j-1)/2+m)}}{\gG\bigpar{(5i+3j-1)/2)}}
\binom{l-j+m}{m}
 \gc_{i,j}\gc_{k-i-m,l-j+m}.
\end{multline}
  \end{romenumerate}
\end{theorem}

\begin{proof}
  Immediate from Lemmas \ref{LedXY} and \ref{Lsing} with the definitions
\begin{align}
  \ga_{k,l}&:=\frac{2^{4k+3l-2}}{k!\,l!}a_{k,l},
\\
  \gb_{k,l}&:=\frac{2^{5k+3l-2}}{k!\,l!}b_{k,l},
\\
  \gc_{k,l}&:=\frac{2^{5k+3l-2}}{k!\,l!}c_{k,l}.
\end{align}
The choice $\ga_{0,0}=\gb_{0,0}=\gc_{0,0}:=-1/2$ satisfies both
\eqref{mxy123}, \eqref{mxy213}, \eqref{mxy231} for $k=l=0$ (trivially)
and the recursions
  \eqref{mxy123r}, \eqref{mxy213r}, \eqref{mxy231r} for $k+l=1$,
\cf{} \refR{R00}.
\end{proof}

Note that when we have proved \refT{Tmain}, 
it follows
that the limits in 
\eqref{mxy123}, \eqref{mxy213}, \eqref{mxy231}
are equal to the moments $\E\bigpar{\gL_{123}^k\gL_{12}^l}$, etc.

\begin{remark}
The number  $\gb_{k,l}$ in \refT{Tmom3}\ref{ledmxy213} 
satisfy the same recursion as 
$\go^*_{l,k} $ in \cite{SJ146}, and thus
$\gb_{k,l}=\go^*_{l,k} $. Indeed they both appear in similar moment
formulas, and the equality is explained by the identities in \refR{Rwiener}
below. 
\end{remark}

In the same way it is possible to find mixed moments of these variables,
first for $\td$ and then (asymptotically, or exact) for $T_n$. 
We give only an example.

\begin{example}\label{Ecov3}
  Let $(V_1,V_2,V_3) = (X_{213}(T_n),X_{231}(T_n),X_{312}(T_n))$ be the
 three  random variables in \eqref{e213}; recall that these have equal mean. 
Using the recursions \eqref{r213}--\eqref{r312}, the method in the proof of
\refL{Lmom1} yields $\ed (V_iV_j)$ as polynomials in $\gdx1$ of degree 9. 
After calculating the leading coefficients (we omit the details), 
we obtain from \refL{Lsing}, in
matrix notation, 
\begin{equation}\label{ecov3}
\bigpar{  n^{-5}\E (V_i V_j)}_{i,j=1}^3 \to \frac{1}{840}
  \begin{pmatrix}
	 49& 42& 42\\  
42 & 43 & 41\\
42 & 41 & 43
  \end{pmatrix}
.
\end{equation}
\end{example}

\section{Brownian functionals}\label{Sbrown} 

Given a binary tree $T$, let $h(v)=h(v;T)$ be the \emph{height} (also called
depth) of a vertex $v\in T$, defined as the distance to the root. 
Thus $h(v)$ is the number of
ancestors of $v$. We define also the \emph{left height} $h_L(v)$ 
as the number of ancestors $w$ of $v$ such that $v$
belongs to the left subtree of $w$, and similarly
the \emph{right height} $h_R(v)$. Equivalently, $h_L(v)$ is the number of
left steps in the path to $v$.

Define the \emph{profile} of a binary tree $T$ as the sequence
$h(v_1),\dots,h(v_n)$, where $v_1,\dots,v_n$ are the vertices of $T$ in
\emph{inorder}; recall that the inorder is defined recursively by taking
first the vertices of $T_L$, then the root and then the vertices of $T_R$
\cite[Section 2.3.1]{KnuthI}.
We write $h(i)=h(v_i)$ and regard $h$ as a function both on the
vertex set of $T$ and on $[n]$. We further define, for $1\le i\le j\le n$,
\begin{equation}
  h([i,j]):=\min_{l\in[i,j]} h(l).
\end{equation}

It is well known that for the random binary tree $T_n$, the height $h(v)$ is
typically of the order $n\qq$. For example, if
$H(T_n):=\max_{v\in T_n} h(v)$  
is the  \emph{height} of $T_n$, then
$H(T_n)/n\qq$ converges in distribution as \ntoo{} (\eg{} as a consequence
of \refL{Ld} below, see \cite{AldousII}).
Moreover, if we normalize the profile by defining 
\begin{equation}\label{tdx}
  \tdx(x)=\tdx(x;T_n):=n\qqw h\bigpar{\floor{nx}+1;T_n}
\end{equation}
(with $\tdx(1)=0$),
which is a function $\oi\to[0,\infty)$, 
then the random function $\tdx(x;T_n)$ converges in distribution to 
the standard normalized Brownian excusion $\ex(x)$, up to a constant factor,
as stated in the following lemma, in principle due to  \citet{AldousIII}.
(Informally, $\ex$ can be seen as Brownian motion on $\oi$ conditioned on
$\ex(x)\ge0$ and $\ex(1)=\ex(0)=0$. 
For formal treatments, 
see \eg{} \cite{DurrettIM} and \cite{RY}.)

\begin{lemma}\label{Ld}
  As \ntoo, $\tdx(x;T_n)\dto 2\qqc \ex(x)$.
\end{lemma}

\begin{remark}\label{Rskorohod}
 The convergence in \refL{Ld} is in the space $D\oi$ of right-continous
 functions with 
 left limits. (We could have defined $\tdx$ as a continuous function
 instead, using linear interpolation of $h(i)$ between integers, 
with no other  essential differences below, and then
 the convergence would have been in $C\oi$.)
For a full technical discussion of convergence in distribution in $D\oi$ or
$C\oi$, see \eg{} \cite{Billingsley}.
For our purposes, we may avoid technicalities by 
the Skorohod representation theorem \cite[Theorem 4.30]{Kallenberg},
which shows that we may assume that the random trees $T_n$ for different
$n$, and $\ex$, are coupled such that
the conclusion $\tdx(x;T_n)\dto 2\qqc \ex(x)$
holds \as{}, uniformly for $x\in\oi$,
\ie, 
$\sup_{x\in\oi}|\tdx(x;T_n)- 2\qqc \ex(x)|\to0$ a.s.
\end{remark}

\begin{proof}
As said above, this is in principle due to \citet{AldousIII}. More
precisely, Aldous considered the \emph{\dfw} on $T_n$, which is the sequence
of vertices $w_0,\dots,w_{2n-2}$ obtained by walking along the ``outside of
the tree'', with $w_0=w_{2n-2}=o$, the root, and beginning with the left
subtree (if any), see \eg{} \cite[Section 4.1.1]{Drmota}. 
Define $f(i):=h(w_i)$ and the normalized version
$\tf(x):=n\qqw f(\floor{2n x})$ for $x\in\oi$ (with $f(2n-1)=f(2n)=0$ for
completeness). 
\citet[Theorem 23]{AldousIII} proved (in greater generality) that  then
$\tf(x)\dto \cex(x)$.

Some variations (and a new proof) were given by \citet{MM}, including a
version with process of heights of the vertices taken in \df{} order
(first the root, then $T_L$, then $T_R$). In the present paper we use
instead the inorder, but the argument in \cite{MM} is easily adapted to this
case too, as follows.

Consider a vertex $v$ in a binary tree $T$. Let $T_L(v)$ and $T_R(v)$ denote
the left and right subtrees of $v$, and let $\cP_v$ be 
the set of the ancestors of $v$
(\ie, the path from the root to $v$, except $v$ itself).
It is easily seen that the vertices that come before $i$ in the inorder are
(i) the set $\cP_{v,R}:=\set{w\in\cP_v:v\in T_R(w)}$ and (ii)
$\cL_v:=\bigcup_{w\in\cP_{v,R}\cup\set v} T_L(w)$.
Hence, $v=v_i$, where
\begin{equation}
  i=1+h_R(v)+|\cL_v|.
\end{equation}
Similarly, since it takes the \dfw{} $2m$ steps to visit a subtree of size
$m$, it is easily seen that if
\begin{equation}
  j:=h(v)+2|\cL_v|,
\end{equation}
then $w_j=v=v_i$. Note that
\begin{equation}\label{2ij}
  |2i-j| =\bigabs{2+2h_R(v)-h(v)}
=\bigabs{2+h_R(v)-h_L(v)}
\le 2+H.
\end{equation}

Now consider again $T_n$.
Let $x\in [0,1)$ and let $i:=\floor{nx}+1$. Find the corresponding vertex
  $v_i\in T_n$ and define $j$ as above, and $y:=j/(2n)$. Then
  \begin{equation}\label{qk}
\tdx(x)
=n\qqw h(v_i)
=n\qqw h(w_{j})
=n\qqw h(w_{{2ny}})
=	\tf(y)
  \end{equation}
and, by \eqref{2ij},
\begin{equation}\label{ql}
  |x-y|\le 
\Bigabs{x-\frac in}+
\frac{|2i-j|}{2n}
\le\frac{4+H}{2n}.
\end{equation}
By the result by \citet{AldousIII} and \refR{Rskorohod}, we may assume that
$\sup|\tf(x)-\cex(x)|\to0$ \as{} as \ntoo.
By \eqref{qk} and \eqref{ql},
\begin{equation*}
  \begin{split}
|\tdx(x)-\cex(x)|	
&=
|\tf(y)-\cex(x)|	
\\&
\le
|\tf(y)-\cex(y)|	
+ 2\qqc|\ex(y)-\ex(x)|	
\\&
\le
\sup_y|\tf(y)-\cex(y)|	
+ 2\qqc\sup_{|x-y|\le (H+4)/2n}|\ex(y)-\ex(x)|.	
  \end{split}
\end{equation*}
The \rhs{} does not depend on $x$ and tends to 0 a.s., by the result of
\citet{AldousIII}, its immediate consequence $H/n\to0$, and the continuity of
$\ex$. 
\end{proof}

Actually, we need the corresponding result for the left height $h_L$. We
define, in analogy with \eqref{tdx},
\begin{equation}\label{tdl}
  \tdl(x)=\tdx(l;T_n):=n\qqw h_L\bigpar{\floor{nx}+1;T_n}.
\end{equation}

The following version of \refL{Ld} is in principle due to
\citet{Marckert-rotation}. 

\begin{lemma}\label{Ldl}
  As \ntoo, $\tdl(x;T_n)\dto \clex(x)$.
\end{lemma}

\begin{proof}
  \citet{Marckert-rotation} proved this for the \df{} order; and, moreover,
  that
  \begin{equation}
	\label{jfm}
n\qqw\max_{v\in T_n}|h_L(v)-h_R(v)|=n\qqw\max_{v\in T_n}|2h_L(v)-h(v)|\pto0.
  \end{equation}
The result follows by \refL{Ld} and \eqref{jfm}.
\end{proof}

\begin{remark}
  It is known that the maxima in \eqref{jfm} actually are of the order
  $n\qqqq$, see \eg{} \cite{Marckert-rotation} and \cite{SJ185} for further
  results. 
\end{remark}

We return to permutations.
Let $\pi\in\fSnzzz$ be a 132-avoiding permutation and let $T$ be the
correspondng binary tree defined in \refS{Sbinary}.
Label the vertices by the corresponding elements of $\pi$. 
(Thus the root is labelled by the maximum element $\pi_\ell=n$.)
The inorder on $T$ corresponds to the standard order on the index set $[n]$;
thus, the permutation $\pi$ can be recovered by taking the labels of $T$ in
inorder. (This is why we need the inorder above.)

Define a partial order on the vertices of $T$ by $v\prec w$ if $v$ is an
ancestor of $w$, \ie{} lies on the path from the root to $w$. 
For two vertices $v$ and $w$, we let $v\wedge w$ be their last common
ancestor (which is their greatest lower bound in this order).

Let, as above, $v_1,\dots,v_n$ be the vertices of $T$ in inorder; thus $v_i$
is labelled by $\pi_i$. 
Consider a pair of distinct $i,j\in[n]$.
It follows  from the construction of $T$ that 
if $v_i\prec v_j$, then $\pi_i>\pi_j$. Symmetrically, 
if $v_j\prec v_i$, then $\pi_i<\pi_j$.
If neither holds, and $i<j$, then there exists a last common ancestor $v_l$
and then $i<l<j$ and $\pi_l>\pi_i>\pi_j$. Consequently, assuming $i<j$, we
have
\begin{equation}
  \label{ij}
\pi_i<\pi_j \iff v_j\prec v_i.
\end{equation}

\begin{theorem}\label{Tpsi}
  Let $\gs\in\fSkzzz$ with $k\ge1$.
Then there exists a continuous functional 
$\Psi_\gs$ on $\coi$ such that
$
n^{-\gl(\gs)/2}\ns(\pinzzz)\dto \Psi_\gs(\ex)
$ 
as \ntoo; 
furthermore, 
$\Psi_\gs(\ex)>0$ a.s.

Moreover,
this holds jointly for all $\gs$.
\end{theorem}

\begin{proof}
  We say that an index $i\in[k]$ is \emph{black} if either $i=1$ or
  $\gs_{i-1}>\gs_i$. (I.e., $i-1$ is 0 or a descent.)
Otherwise, $i$ is \emph{white}.
Let $B$ be the set of black indices, and $W$ the set of white indices.
Thus $|B|=D(\gs)$ and $|W|=|\gs|-D(\gs)$.

\emph{Claim}: If $i<j$ and $j$ is a black index, so $\gs_{j-1}>\gs_j$, then
$\gs_i>\gs_j$, since otherwise $\gs_i\gs_{j-1}\gs_j$ would be an occurence
of 132 in $\gs$.

Let $\nu_1\dotsm\nu_k$ be a sequence with $1\le\nu_1<\dotsm<\nu_k\le n$ and
let us investigate whether
$\pi_{\nu_1}\dotsm\pi_{\nu_k}$ is an occurrence of $\gs$ in $\pi$.
Write, for convenience, $\vv_{i}=v_{\nu_i}$, the vertex in $T$ 
with label $\pi_{\nu_i}$. We say that $\nu_i$, or $\vv_i$, is black or white
if $i$ is.

We first consider $\nu_i$, or equivalently $\vv_i$, for the black indices $i$.
By the claim above, if $i$ and $j$ are black indices with $i<j$, then
$\gs_i>\gs_j$ and thus we require $\pi_{\nu_i}>\pi_{\nu_j}$, which by
\eqref{ij} is equivalent to $\vv_j\not\prec\vv_i$.
The only condition for the black vertices $\vv_i$ is thus that they are in
increasing inorder and none is an ancestor of a previous one.

We then consider $\nu_i$ for the white indices, in order from left to
right. For each white $j$ the conditions are as follows, by \eqref{ij}
and the claim above.
Let $U_j:=\set{i<j:\gs_i<\gs_j}$ and note that $j-1\in U_j$ since $j$ is
white.
\begin{romenumerate}
\item \label{wa}
$\nu_j>\nu_{i}$ for $i<j$.
\item \label{wb}
$\vv_j\prec\vv_i$ for $i\in U_j$.
\item \label{wc}
$\vv_j\not\prec\vv_i$ for $i\in [j-1]\setminus U_j$.
\item \label{wd}
$\nu_j<\nu_i$ and $\vv_i\not\prec\vv_j$ for every black $i>j$.
\end{romenumerate}

Furthermore, let
$b=b(j)$ be the largest black index in $[j-1]$.

The index $b\in U_j$ so by \ref{wb}, $\vv_j\prec\vv_b$, \ie, $\vv_j$ is on
the path from the root to $\vv_b$. 
Moreover, by \ref{wa}, $\nu_j>\nu_b$, so
$\vv_j$ comes after $\vv_b$ in the inorder; this means that the next step
from $\vv_j$ on the path to $\vv_b$ is to the left.
The number of such $\vv_j$ (ignoring the other conditions) is $h_L(\vv_b)$.
For such $\vv_j$, the condition \ref{wb} that $\vv_j\prec\vv_i$ for $i\in
U_j$ is equivalent to $\vv_j\prec\vv_i\wedge\vv_b$, and thus to,
if $i\le b$,
\begin{equation}\label{q1}
  h_L(\nu_j)=h_L(\vv_j) < h_L(\vv_i\wedge\vv_b)=h_L([\nu_i,\nu_b]);
\end{equation}
if $i>b$, so also $\vv_i$ is on the path to $\vv_b$, the condition
is simply 
\begin{equation}
  \label{q2}
h_L(\nu_j)<h_L(\nu_i). 
\end{equation}

For $i\in[j-1]\setminus U_j$, which implies $i<b$,
\ref{wc} conversely requires
\begin{equation}\label{q3}
  h_L(\nu_j) \ge h_L([\nu_i,\nu_b]).
\end{equation}

In \ref{wd}, for black $i>j$, the condition 
$\vv_i\not\prec\vv_j$ is redundant, since we already know $\vv_j\prec\vv_b$
and $\vv_i\not\prec\vv_b$ (both $b$ and $i$ are black).
Since  $\vv_j\prec\vv_b$ and $\nu_b<\nu_j$, we see also that $\nu_j<\nu_i$ 
implies $\vv_j\succeq\vv_b\wedge\vv_i$. 
(If $\vv_j\prec\vv_b\wedge\vv_i$,
then $\vv_b$ and $\vv_i$ are on the same side of $\vv_j$.)
This means 
\begin{equation}\label{q4}
  h_L(\nu_j) \ge h_L([\nu_b,\nu_i]).
\end{equation}
Conversely, \eqref{q1}--\eqref{q4} are also sufficient for \ref{wa}--\ref{wd}.
(Note that \eqref{q4} implies that 
$\vv_j\in T_L(\vv_b\wedge\vv_i)\cup\set{\vv_b\wedge\vv_i}$
and $\vv_i\in T_R(\vv_b\wedge\vv_i)$; thus $\nu_j<\nu_i$.
Similarly, \ref{wa} follows from \eqref{q1}--\eqref{q3}.)

Consequently, having chosen the black vertices, we have to choose $\nu_j$
for the white indices $j$ such that 
$\vv_j$ is on the path from the root to $\vv_{b(j)}$, with a left step next,
and 
\eqref{q1}--\eqref{q4} hold.

Let us count.
We choose first the black vertices, one by one. There are $\binom n{D(\gs)}$
choices of $\set{\nu_i:i\in B}$. Of these, the condition
$\vv_j\not\prec\vv_i$ for $i<j$ forbids only $O(H)$ choices for each $j$
(where $H=H(T)$ is the height), and thus
$O\bigpar{n^{D(\gs)-1}H}$ choices of the black vertices. We will simply
ignore this restriction, introducing an error that will be negligible.

For each choice of black vertices, we then choose the white vertices
$\vv_i$. By the conditions above, there are at most $H$ choices for each
white $\vv_i$, and thus at most $H^{|W|}=H^{k-D(\gs)}$ choices for
$\set{\vv_i:i\in W}$. More precisely, this number is a polynomial 
$\Phi=\Phi_\gs$ of degree
$k-D(\gs)$ in the numbers $h_L(\nu_i)$ and $h_L([\nu_i,\nu_j])$, $i,j\in
B$. We will not attempt to give an exact description of this polynomial in
general, but we give after the proof a few examples that will illustrate the
construction, and it should be clear that similar constructions hold
in general.

Let $B=\set{b_1,\dots,b_D}$ where $D=D(\gs)$.
We regard $\Phi$ as a functional of the left profile $h_L$ and the black
indices $\nu_{b_1},\dots,\nu_{b_D}$,  and obtain
\begin{equation}\label{jul}
  \ns(\pi) 
= \sum_{\nu_{b_1}<\dotsm<\nu_{b_D}} \Phi\bigpar{h_L;\nu_{b_1},\dots,\nu_{b_D}}
+O\bigpar{n^{D-1}H\cdot H^{k-D}},
\end{equation}
where the error term comes from including also forbidden sets of black
vertices. 

We now use \refL{Ldl}. By the Skorohod representation theorem, see
\refR{Rskorohod}, we may assume that 
$\tdl(x)\to 2\qq\ex(x)$ uniformly on $\oi$ as \ntoo.
In particular this implies that $n\qqw\max_i h_L(i)=\sup_{x\in\oi}\tdl(x)=O(1)$.
(The implicit constant is random but does not depend on $n$.)
Similarly, by \refL{Ld}, we may assume
$n\qqw H=n\qqw\max_i h(i)=O(1)$, \ie, $H=O(n\qq)$.

Letting $\Phi'$ be the leading terms in $\Phi$, which are homogeneous of
degree $|W|=k-D$, 
and letting $\Phix$ be the corresponding functional for functions on \oi,
we then obtain from \eqref{jul},
\begin{equation*}
  \begin{split}
  \ns(\pinzzz) &
= \sum_{\nu_{b_1}<\dotsm<\nu_{b_D}} \Phi'\bigpar{h_L;\nu_{b_1},\dots,\nu_{b_D}}
+O\bigpar{n^D H^{k-D-1}}
\\&\hskip16em
+O\bigpar{n^{D-1}H^{1+k-D}}  
\\&
= n^{(k-D)/2}\sum_{0\le i_1<\dotsm<{i_D}\le n-1} 
\Phix\bigpar{\tdl;{i_1/n},\dots,{i_D/n}}
+O\bigpar{ n^{(D+k-1)/2}}
\\&
= n^{(k+D)/2}\int_{0\le x_1<\dotsm<{x_D}\le 1} 
\Phix\bigpar{\tdl;{x_1},\dots,{x_D}}\dd x_1\dotsm \dd x_D
\\&\hskip16em
+O\bigpar{ n^{(D+k-1)/2}}.
  \end{split}
\end{equation*}
We define, for a function $f$ on \oi,
\begin{equation}\label{psis}
  \Psis(f):=
2^{(k-D)/2}\int_{0\le x_1<\dotsm<{x_D}\le 1} 
\Phix\bigpar{f;{x_1},\dots,{x_D}}\dd x_1\dotsm \dd x_D
\end{equation}
and have thus, by the uniform convergence $\tdl\to 2\qq \ex$,
\begin{equation}
n^{-(k+D)/2}  \ns(\pinzzz)
=\Psis(2\qqw\tdl)+o(1)
\to\Psis(\ex).
\end{equation}
It is obvious that $\Psis(\ex)>0$ a.s.
This completes the proof. 
\end{proof}

\begin{example}\label{Epsi12}
  $\gs=12$.
$B=\set1$. For every choice of the black $\nu_1$, there are $h_L(\nu_1)$
  choices of the white $\nu_2$. Hence \eqref{jul} is simply
$\ns(\pi)=\sum_{\nu=1}^n h_L(\nu)$ and  $\Phi(h_L;\nu)=h_L(\nu)$.
Consequently $\Phix(f;x)=f(x)$ and \eqref{psis} yields
\begin{equation}\label{gl12}
  \gL_{12}=X_{12}(T) = \Psi_{12}(\ex)=\sqrt2\intoi \ex(x)\dd x.
\end{equation}
As said in \refS{Sresults},
this is, apart from the factor $\sqrt2$, the well-known Brownian excursion
area.
\end{example}

\begin{example}
  $\gs=123$.
$B=\set1$. 
Both $\vv_2$ and $\vv_3$ are on the path to $\vv_1$, both with the next step
left, and with $\vv_3\prec\vv_2$.
There are $\binom{h_L(\nu_1)}2$ choices of them for each $\nu_1$. Hence, 
\begin{equation}
n_{123}(\pi)=X_{123}(T)=\sum_{\nu=1}^n \binom{h_L(\nu)}2
\end{equation}
which leads to $\Phix(f;x)=\frac12f(x)^2$ and
\begin{equation}
  \gL_{123} = \Psi_{123}(\ex)=\intoi \ex(x)^2 \dd x.
\end{equation}
The joint distribution of the random variables $\intoi\ex$ (see
\refE{Epsi12}) and $\intoi\ex^2$ 
have been studied by \citet{Nguyen}, who found a recursion for mixed moments
equivalent to \refT{Tmom3}\ref{ledmxy123}. He also found the Laplace
transform
$\E e^{-t\gL_{123}}=\bigpar{\sqrt{2t}/\sinh(\sqrt{2t})}^{3/2}$,
which shows that  $\gL_{123}$ has the distribution denoted $S_{3/2}$ in 
\citet{BianePY}, see in particular \cite[Secion 4.4]{BianePY}
(and recall that $\ex$ can be seen as a 3-dimensional Bessel bridge).
Equivalently, $\gL_{123}$ has the moment generating function
\begin{equation}\label{mgf123}
\E e^{t\gL_{123}}=\parfrac{\sqrt{2t}}{\sin(\sqrt{2t})}^{3/2},  
\qquad \Re t < \frac{\pi^2}2.
\end{equation}
\end{example}

\begin{example}\label{Epsi1...k}
More generally, for $\gs=1\dotsm k$, any $k\ge1$, 
\begin{equation}
n_{1\dotsm k}(\pi)=X_{\gsik}(T)=\sum_{\nu=1}^n \binom{h_L(\nu)}{k-1}
\end{equation}
and
\begin{equation}
  \gL_{1\dotsm k} = \Psi_{\gsik}(\ex)
=\frac{2^{(k-1)/2}}{(k-1)!} \intoi \ex(x)^{k-1} \dd x.
\end{equation}
Thus, if $Z_k:=\intoi \ex(x)^k$, the average of the $k$:th power of the
Brownian excursion, then 
$\gL_{1\dotsm k} =c_k Z_{k-1}$ with $c_k=2^{(k-1)/2}/(k-1)!$.
The random variables $Z_k$ have been studied by 
\citet{Richard}; in particular, \cite{Richard} gives a recursion formula for
the mixed moments, which is equivalent to our recursion implicit in the
proof of \refL{Lmom1} for this case.

Note that, by \Holder's inequality, $Z_k\ge Z_1^k$ for every $k$; hence
by the known asymptotics for moments of the Brownian excursion area $Z_1$,
see \eg{} \cite{SJ201},
as \rtoo,
\begin{equation}\label{ulk1}
  \E Z_k^r \ge \E Z_1^{kr} \sim 3\sqrt2 kr  \Bigpar{\frac{kr}{12 e}}^{kr/2}.
\end{equation}
More precisely, it follows from \cite[Theorem 2.1]{SJ197} 
(applied to $Z_k^{1/k}$) that for every fixed $k\ge1$, as \rtoo,
\begin{equation}\label{ulk2}
   \bigpar{\E Z_k^r}^{1/r} \sim z_k r^{k/2}
\end{equation}
where $z_k>0$ is a constant given by
\begin{equation*}
  z_k:= \Bigparfrac ke^{k/2}\max\biggset{\intoi f(x)^k:f(0)=f(1)=0
\text{ and }
\intoi (f'(x))^2\le1}.
\end{equation*}
(We have $z_1=1/\sqrt{12 e}$ and $z_2=2/(e\pi^2)$; we do not know $z_k$ for
$k>2$.) 

It follows from \eqref{ulk1}--\eqref{ulk2}
that the \mgf{} $\E e^{tZ_k}$ of $Z_k$ is an entire function 
for $k=1$ (see further \cite{SJ201}), but has a finite radius of convergence
for $k=2$ and diverges for all $t>0$ when $k\ge3$.
(The claim in \cite[Theorem 1.2]{Richard} that $Z_1,\dots,Z_M$ have an
entire \mgf{} is thus incorrect. For $Z_2$ this is also seen by the explicit
formula \eqref{mgf123}.) 

Moreover, 
the Carleman condition
(in its weaker form for nonnegative random variables,  
see \eg{} \cite[Section 4.10]{Gut})
\begin{equation*}
\sum_m \bigpar{ \E Z_k^m}^{-1/2m}=\infty 
\end{equation*}
holds by \eqref{ulk2} for $k\le4$ but not for $k\ge5$.
Although the Carleman condition is only sufficient for a distribution to be
determined by its moments,
this strongly suggests that the distribution of $Z_k$, and thus
$\gL_{1\dots(k+1)}$, is \emph{not} determined by its moments if $k$ is large
enough. 
\end{example}

\begin{example}
  $\gs=213$.
$B=\set{1,2}$. 
Given $\vv_1$ and $\vv_2$, the white vertex $\vv_3$ has to be on the path to
$\vv_1\wedge\vv_2$. 
There are ${h_L(\vv_1\wedge\vv_2)}=h_L([\nu_1,\nu_2])$ choices, and thus
\begin{equation}
n_{213}(\pi)=X_{213}(T)=\sum_{\nu_1<\nu_2} h_L([\nu_1,\nu_2]) +O(nH^2)
\end{equation}
which leads to
\begin{equation}\label{gl213}
\gL_{213} = \Psi_{213}(\ex)= \sqrt2\iint_{0\le x<y\le1} \ex([x,y]) \dd x \dd y.
\end{equation}
\end{example}

\begin{example}
  $\gs=231$.
$B=\set{1,3}$. 
Given $\vv_1$ and $\vv_3$, the white vertex $\vv_2$ has to be on the path to
$\vv_1$ but not to $\vv_3$. 
There are $h_L(\vv_1)-{h_L(\vv_1\wedge\vv_3)}-1$ choices, and thus
\begin{equation}
n_{231}(\pi)=X_{231}(T)=\sum_{\nu_1<\nu_3}
\bigpar{h_L(\vv_1)- h_L([\nu_1,\nu_3])-1} +O(nH^2)
\end{equation}
which leads to
\begin{equation}\label{gl231}
  \gL_{231} = \Psi_{231}(\ex)= \sqrt2\iint_{0\le x<y\le1}
 \bigpar{\ex(x)- \ex([x,y])} \dd x \dd y.
\end{equation}
\end{example}

\begin{example}
  $\gs=312$.
$B=\set{1,2}$. 
Given $\vv_1$ and $\vv_2$, the white vertex $\vv_3$ has to be on the path to
$\vv_2$ but not to $\vv_1$. 
Thus
\begin{equation}
n_{312}(\pi)=X_{312}(T)=\sum_{\nu_1<\nu_2}
\bigpar{h_L(\vv_2)- h_L([\nu_1,\nu_2])-1} +O(nH^2)
\end{equation}
which leads to
\begin{equation}\label{gl312}
  \gL_{312} = \Psi_{312}(\ex)= \sqrt2\iint_{0\le x<y\le1}
 \bigpar{\ex(y)- \ex([x,y])} \dd x \dd y.
\end{equation}

Note that the equality in distribution
$\gL_{231}\eqd \gL_{312}$ here is immediate by 
\eqref{gl231}, \eqref{gl312} and
the
symmetry $\ex(x)\eqd\ex(1-x)$ of the Brownian excursion.
However, we see also that $\gL_{231}$ and $\gL_{312}$ differ as random
variables, which means that the joint distribution of $n_{231}$ and
$n_{312}$ 
does not have degenerate (one-dimensional) 
asymptotic distribution.
Cf.~the second moments in \eqref{ecov3}. 
\end{example}

Furthermore, note that the identity \eqref{rom} also can be seen from 
\eqref{gl12}, \eqref{gl213}, \eqref{gl231}, \eqref{gl312}.

\begin{remark}\label{Rwiener}
\citet{SJ146} studied some functionals of random trees and found as limits
in distribution three functionals of Browninan excursion, there denoted
$\xi$, $\eta$, $\zeta$. $\xi$ is simply twice the Brownian excursion area,
so by \eqref{gl12}, $\xi=\sqrt2 \gL_{12}$. Furthermore,
$\eta$ is 4 times the integral in \eqref{gl213}, and thus
$\eta=2^{3/2}\gL_{213}$. Finally,
$\zeta=\xi-\eta=2\qq(\gL_{231}+\gL_{312})$, by  
\eqref{rom} or by comparing the formula in \cite{SJ146} to
\eqref{gl231} and \eqref{gl312}.
\end{remark}

\begin{example}\label{Epsik...1}
  $\gs=\gski$. This is the trivial case when all vertices are black, so
$\Phi=1$ is constant. Thus also $\Phix=1$ and \eqref{psis} yields
$\Psi_{\gski}=1/k!$, in accordance with \refT{Tmain} and \refR{Rk1}. 
\end{example}

Although the expressions get increasingly more complicated, it is clear that
there are of the same nature for every $\gs$. In particular, 
except for the case $\gs=\gski$, see
\refE{Epsik...1}, $\Psis(\ex)$ is non-degenerate, \ie, not \as{} constant.

\section{Proof of \refT{Tmain}}

\refT{Tpsi} shows
the existence of limits $\gL_\gs=\Psis(\ex)$ such that \eqref{tmaind} holds,
jointly for all $\gs\in\fSxzzz$. Furthermore, $\gL_\gs>0$ \as{} and $\gL_\gs$ is
non-degenerate except in the case $\gs=\gski$.

On the other hand, Lemmas \ref{Lmom1} and \ref{Lsing} show that
for any $\gspp1,\dots,\gspp{M}$,
\begin{equation}\label{mompf}
 n^{-\sum_\nu\gl(\gspp\nu)/2}  \E\bigpar{n_{\gspp 1}\dotsm n_{\gspp M}(\pinzzz)}
\to A_{\gspp1,\dots,\gspp M}
\end{equation}
for some constant $A_{\gspp1,\dots,\gspp M} <\infty$.
As is well-known, convergence of all moments implies that all 
products $n_{\gspp 1}\dotsm n_{\gspp M}(\pinzzz)$ are uniformly integrable,
and thus the limits of the moments are the moments of the limits $\gL_\gs$.
\cite[Theorems 5.4.2 and 5.5.9]{Gut}.

\begin{remark}
  Note that we cannot use \eqref{mompf} to show the existence of limits
  $\gL_\gs$ in \eqref{tmaind}, since we cannot show that the limit
  distributions are determined by thier moments; on the contrary, we believe
  that they in general are not, see \refE{Epsi1...k}. 
This is one reason for using two different methods in the proof above, one
for the existence of limits in distribution and another for the limits of
moments. 
\end{remark}

\section{Further comments}\label{Sfurther}

As said in \refR{Rbonadom}, \citet{Bona-abscence} has shown that for every
$n$
and any $\gs\in\fSkzzz$,
\begin{equation}
  \label{bonadom}
\E n_{1\dotsm k}(\pinzzz) \le \E n_{\gs}(\pinzzz)
\le \E n_{k\dotsm 1}(\pinzzz).
\end{equation}
We use here the recursion \refL{LR} to show a more general result.

Define a partial order on each $\fSkzzz$ by
  \begin{equation*}
	\gs\prec\gs' 
\iff \set{(i,j):i<j\text{ and }\gs_i>\gs_j}
\subseteq
 \set{(i,j):i<j\text{ and }\gs'_i>\gs'_j}.
  \end{equation*}
\begin{theorem}
If\/ $|\gs|=|\gs'|$ and 
$\gs\prec\gs'$, then $\E n_{\gs}(\pinzzz)  \le \E n_{\gs'}(\pinzzz) $ for
every $n\ge1$.
\end{theorem}

Note that $1\dotsm k$ is minimal and $k\dotsm1$ is maximal in the partial
order $\prec$, so \eqref{bonadom} follows immediately.

\begin{proof}
  We use induction on $n$. The case $n=1$ is trivial.

Condition on the value of the maximal index $\ell$ in $\pi=\pinzzz$.
Given $\ell$, $\pi_L$ and $\pi_R$ are independent uniformly random elements
of $\fS_{\ell-1}(\zzz)$ and $\fS_{n-\ell}(\zzz)$ respectively.
Furthermore,
$\gs\prec\gs'$ implies that 
$\gs_1\dotsm\gs_q \prec \gs'_1\dotsm\gs'_q$
and 
$\gs_{q+1}\dotsm\gs_k \prec \gs'_{q+1}\dotsm\gs'_k$ for every $q\in[k]$, and
also $\gD_{\gs}\subseteq\gD_{\gs'}$. 

Using \eqref{lr} for both $\gs$ and $\gs'$ and taking the conditional
expectations, it follows, by this and the induction hypothesis, that
\begin{equation}
 \E \bigpar{n_\gs(\pinzzz) \mid \ell }
\le
\E \bigpar{n_{\gs'}(\pinzzz) \mid \ell} 
\end{equation}
for every value of $\ell\in[n]$.
Taking the expectation we obtain
$ \E n_\gs(\pinzzz)\le\E n_{\gs'}(\pinzzz)$, completing the induction step.
\end{proof}

\citet{Rudolph} has a general result, and a conjecture, for the 
related problem of when there is equality
$\E n_{\gs}(\pinzzz)=\E n_{\gs'}(\pinzzz)$ for all $n$.
It seems possible that \refL{LR} can be used to prove, and perhaps improve,
her results too, but we have not attempted this.

\newcommand\AAP{\emph{Adv. Appl. Probab.} }
\newcommand\JAP{\emph{J. Appl. Probab.} }
\newcommand\JAMS{\emph{J. \AMS} }
\newcommand\MAMS{\emph{Memoirs \AMS} }
\newcommand\PAMS{\emph{Proc. \AMS} }
\newcommand\TAMS{\emph{Trans. \AMS} }
\newcommand\AnnMS{\emph{Ann. Math. Statist.} }
\newcommand\AnnPr{\emph{Ann. Probab.} }
\newcommand\CPC{\emph{Combin. Probab. Comput.} }
\newcommand\JMAA{\emph{J. Math. Anal. Appl.} }
\newcommand\RSA{\emph{Random Struct. Alg.} }
\newcommand\ZW{\emph{Z. Wahrsch. Verw. Gebiete} }
\newcommand\DMTCS{\jour{Discr. Math. Theor. Comput. Sci.} }

\newcommand\AMS{Amer. Math. Soc.}
\newcommand\Springer{Springer-Verlag}
\newcommand\Wiley{Wiley}

\newcommand\vol{\textbf}
\newcommand\jour{\emph}
\newcommand\book{\emph}
\newcommand\inbook{\emph}
\def\no#1#2,{\unskip#2, no. #1,} 
\newcommand\toappear{\unskip, to appear}

\newcommand\arxiv[1]{\url{arXiv:#1.}}
\newcommand\arXiv{\arxiv}

\def\nobibitem#1\par{}

\end{document}